\documentclass[a4paper,12pt]{amsart}
\usepackage{amsmath}
\usepackage{amsthm}
\usepackage{amssymb}
\usepackage{amscd}

\newtheorem{thm}{Theorem}[section]

\newtheorem{prop}[thm]{Proposition}
\newtheorem{lemma}[thm]{Lemma}

\newtheorem{remark}[thm]{Remark}

\newcommand{\Pic}{\mathop{\rm Pic}\nolimits}

\DeclareMathOperator{\Hom}{Hom}
\DeclareMathOperator{\Ext}{Ext}
\DeclareMathOperator{\RHom}{RHom}
\DeclareMathOperator{\lotimes}{\otimes^{L}}
\newcommand{\caltor}{\mathop{{\mathcal T\!or}}\nolimits}

\DeclareMathOperator{\Coker}{Coker}
\DeclareMathOperator{\End}{End}
\DeclareMathOperator{\length}{length}

\DeclareMathOperator{\module}{mod}

\DeclareMathOperator{\Gr}{Gr}

\DeclareMathOperator{\ev}{ev}

\makeatletter

\@addtoreset{equation}{section}
\makeatother

\title[
Nef vector bundles on a quadric surface
]{
Nef vector bundles on a quadric surface
with first Chern class $(2,2)$
}
\thanks{
This work was partially supported by 
JSPS KAKENHI (C) Grant Number 21K03158.
}

\author{Masahiro Ohno
}

\address{Graduate School of Informatics and Engineering,
The University of Electro-Communications,
Chofugaoka 1-5-1,
Chofu-shi,
Tokyo, 182-8585 Japan
}

\email{masahiro-ohno@uec.ac.jp}



\begin{document}
\begin{abstract}
We classify nef vector bundles on a smooth quadric surface
with first Chern class $(2,2)$ 
over an algebraically closed field of characteristic zero.
\end{abstract}

\maketitle


\section{Introduction}
In \cite[Sect.~3]{swCompo} and \cite[Sect.~2 Lemmas 1 and 2]{pswnef},
Peternell-Szurek-Wi\'{s}niewski
classified
nef vector bundles on a smooth quadric surface 
$\mathbb{Q}^2$
with first Chern class $(1,1)$
over an algebraically closed field $K$ of characteristic zero.
In \cite[Theorem~7.1]{MR4453350}, we 
provided
a different proof of this classification,
which was based on an analysis with a full strong exceptional collection
of line bundles on $\mathbb{Q}^2$.
Following this approach, we classified those 
with first Chern class $(2,1)$  in \cite[Theorem~0.1]{MR4052950}. 

In this paper, we 
classify nef vector bundles on 
$\mathbb{Q}^2$ with first Chern class $(2,2)$.
The precise statement is as follows:

\begin{thm}\label{Chern(2,2)}
Let $\mathcal{E}$ be a nef vector bundle of rank $r$
on a smooth quadric surface $\mathbb{Q}^2$ over an algebraically closed field $K$ 
of characteristic zero.
Suppose that 
$\det\mathcal{E}\cong \mathcal{O}(2,2)$.
Then 
$\mathcal{E}$ is isomorphic to one of the following vector bundles
or fits in one of the following exact sequences:
\begin{enumerate}
\item[$(1)$]
$\mathcal{O}(2,2)\oplus \mathcal{O}^{\oplus r-1}$;
\item[$(2)$] 
$\mathcal{O}(2,1)\oplus\mathcal{O}(0,1)\oplus  \mathcal{O}^{\oplus r-2}$;
\\
$\mathcal{O}(1,2)\oplus
\mathcal{O}(1,0)\oplus  \mathcal{O}^{\oplus r-2}$;
\\
$($We do not exhibit the cases
obtained by replacing $(a,b)$ with $(b,a)$ in the following:$)$
\item[$(3)$] $\mathcal{O}(1,1)^{\oplus 2}
\oplus  \mathcal{O}^{\oplus r-2}$;
\item[$(4)$] 
$0\to \mathcal{O}
\xrightarrow{\iota}
\mathcal{O}(1,1)\oplus
\mathcal{O}(1,0)\oplus  
\mathcal{O}(0,1)\oplus  
\mathcal{O}^{\oplus r-2}\to \mathcal{E}\to 0$;\\
$($More precisely, either the composite of $\iota$ and the projection 
\[\mathcal{O}(1,1)\oplus
\mathcal{O}(1,0)\oplus  
\mathcal{O}(0,1)\oplus  
\mathcal{O}^{\oplus r-2}\to \mathcal{O}\]
is zero, or 
$\mathcal{E}$ is isomorphic to $\mathcal{O}(1,1)\oplus
\mathcal{O}(1,0)\oplus  
\mathcal{O}(0,1)\oplus  
\mathcal{O}^{\oplus r-3}$.$)$
\item[$(5)$] $0\to \mathcal{O}(-1,-1)\to \mathcal{O}(1,1)\oplus
\mathcal{O}^{\oplus r}\to \mathcal{E}\to 0$;
\item[$(6)$]  $0\to \mathcal{O}^{\oplus 2}\to \mathcal{O}(1,0)^{\oplus 2}\oplus
\mathcal{O}(0,1)^{\oplus 2}\oplus  
\mathcal{O}^{\oplus r-2}\to \mathcal{E}\to 0$;
\\
$($More precisely, 
\begin{enumerate}
\item[$(6$-$1)$] 
$0\to \mathcal{O}\to \mathcal{O}(2,0)\oplus
\mathcal{O}(0,1)^{\oplus 2}\oplus  
\mathcal{O}^{\oplus r-2}\to \mathcal{E}\to 0$;
\\ In this case, we have
\begin{enumerate}
\item[$(6$-$1$-$1)$] $\mathcal{O}(2,0)\oplus\mathcal{O}(0,2)\oplus\mathcal{O}^{\oplus r-2}$;
\item[$(6$-$1$-$2)$] $\mathcal{O}(2,0)\oplus\mathcal{O}(0,1)^{\oplus 2}\oplus  \mathcal{O}^{\oplus r-3}$;
\end{enumerate}
\item[$(6$-$2)$] $\mathcal{O}(1,0)^{\oplus 2}\oplus  
\mathcal{O}(0,1)^{\oplus 2}\oplus  
\mathcal{O}^{\oplus r-4}$$)$
\item[$(6$-$3)$] $0\to \mathcal{O}(0,-1)\to \mathcal{O}(1,0)^{\oplus 2}\oplus  
\mathcal{O}(0,1)\oplus  
\mathcal{O}^{\oplus r-2}\to \mathcal{E}\to 0$;
\end{enumerate}
\item[$(7)$] $0\to \mathcal{O}(-1,-1)
\oplus \mathcal{O}(-1,0)
\oplus \mathcal{O}(0,-1)
\to \mathcal{O}^{\oplus r+3}\to \mathcal{E}\to 0$;
\item[$(8)$] $0\to \mathcal{O}(-1,-2)
\to \mathcal{O}(1,0)\oplus 
\mathcal{O}^{\oplus r}\to \mathcal{E}\to 0$;
\item[$(9)$] $0\to \mathcal{O}(-1,-1)^{\oplus 2}\to   
\mathcal{O}^{\oplus r+2}\to \mathcal{E}\to 0$;
\item[$(10)$] $0\to \mathcal{O}(-2,-2)\to   
\mathcal{O}^{\oplus r+1}\to \mathcal{E}\to 0$;
\item[$(11)$] $0\to \mathcal{O}(-2,-2)\to   
\mathcal{O}^{\oplus r+1}\to \mathcal{E}\to k(p)\to 0$;
\item[$(12)$] $0\to \mathcal{O}(-2,-2)\to   
\mathcal{O}^{\oplus r}\to \mathcal{E}\to \mathcal{O}\to 0$;
\item[$(13)$] $0\to \mathcal{O}(-1,-1)^{\oplus 4}\to   
\mathcal{O}^{\oplus r}\oplus\mathcal{O}(-1,0)^{\oplus 2}
\oplus \mathcal{O}(0,-1)^{\oplus 2}
\to \mathcal{E}\to 0$.
\end{enumerate}
\end{thm}
Note that this list is effective: in each case exists an example.
Under the stronger assumption that $\mathcal{E}$ is globally generated,
Ballico-Huh-Malaspina gave a 
classification of $\mathcal{E}$
with $c_1=(2,2)$
in \cite{MR3319928}.
Note that the vector bundles in Case (11), (12) or (13) are nef but not globally generated.
In fact, the cokernel $\Coker(\ev)$ of the evaluation map 
$\ev:H^0(\mathcal{E})\otimes \mathcal{O}\to \mathcal{E}$
is the structure sheaf of point $p$ in Case (11) 
and of the entire space $\mathbb{Q}^2$
in Case (12),
and, in Case (13),
a torsion-free sheaf of rank one on a divisor $E$ of type $(1,1)$
(see Remark~\ref{support(2,2)}).

Note that if $\mathcal{E}$ is the same as in Theorem~\ref{Chern(2,2)}
then the canonical bundle of the projective space bundle $\mathbb{P}(\mathcal{E})$
is nef.
Furthermore $\mathcal{O}_{\mathbb{P}(\mathcal{E})}(1)^{r+1}=8-c_2$,
where $c_2$ 
is the second Chern class of $\mathcal{E}$.
Hence $\mathbb{P}(\mathcal{E})$ is a weak Fano manifold, 
if $c_2<8$,
namely except for Cases (10), (12) and (13).

Our basic strategy and framework for describing 
$\mathcal{E}$
in Theorem~\ref{Chern(2,2)} 
is to provide a minimal locally free resolution 
of $\mathcal{E}$
in terms of some twists of the full strong exceptional collection 
\[(\mathcal{O},\mathcal{O}(1,0),\mathcal{O}(0,1),\mathcal{O}(1,1))\] 
of line bundles (see \cite{MR4453350} for further details).
In Cases (3), (4), (6), (7), (9) and (13) of Theorem~\ref{Chern(2,2)},
we describe $\mathcal{E}$ in this framework; 
however we do not 
describe every single $\mathcal{E}$ in this framework
in order 
to avoid unnecessarily messy 
descriptions.
For example, we do not describe $\mathcal{O}(2,2)$ in Case~(1)
of Theorem~\ref{Chern(2,2)} in terms of the 
aforementioned 
exceptional collection.
Instead, we 
make the description neat.

We also classify nef vector bundles 
with 
a near-to-maximal degree line bundle as a subsheaf:

\begin{prop}\label{HalfMaximal}
Let $\mathcal{E}$ be a nef vector bundle of rank $r$ on 
a smooth quadric surface $\mathbb{Q}^2$
with first Chern class $c_1=(c_1',c_1'')$.
Suppose that $\Hom (\mathcal{O}(c_1',b),\mathcal{E})\neq 0$
and $\Hom (\mathcal{O}(c_1',b+1),\mathcal{E})=0$ for some integer $b\leq c_1''$.
Then $\mathcal{E}$ fits in the following exact sequence:
\[0\to \mathcal{O}(c_1',b)\to \mathcal{E}\to \pi_2^*\mathcal{F}_2\to 0,\]
where $\pi_2:\mathcal{Q}\to \mathbb{P}^1$ denotes the second projection
and $\mathcal{F}_2$ is a nef vector bundle on $\mathbb{P}^1$ with degree $c_1''-b$. 
In particular, if $b=c_1''$, then 
$\mathcal{E}\cong \mathcal{O}(c_1',c_1'')\oplus \mathcal{O}^{\oplus r-1}$,
and if $b\leq c_1''-1$, then $\mathcal{E}$ fits in the following exact sequence:
\[0\to \mathcal{O}^{\oplus c_1''-b-1}\to 
\mathcal{O}(c_1',b)\oplus \mathcal{O}(0,1)^{\oplus c_1''-b}
\oplus \mathcal{O}^{\oplus r-2}
\to \mathcal{E}\to 0.\]
\end{prop}

\begin{thm}\label{nearToMaximal}
Let $\mathcal{E}$ be a nef vector bundle of rank $r$
on a smooth quadric surface $\mathbb{Q}^2$
with first Chern class $
(c_1',c_1'')$.
Suppose that 
$\Hom (\mathcal{O}(c_1'-1,c_1''),\mathcal{E})
=\Hom (\mathcal{O}(c_1',c_1''-1),\mathcal{E})=0$
and that $\Hom (\mathcal{O}(c_1'-1,c_1''-1),\mathcal{E})\neq 0$.
Then $\mathcal{E}$ satisfies one of the following:
\begin{enumerate}
\item[$(1)$] $\mathcal{E}\cong \mathcal{O}(c_1'-1,c_1''-1)\oplus\mathcal{O}(1,1)\oplus
\mathcal{O}^{\oplus r-2}$;
\item[$(2)$]
$0\to \mathcal{O}\to 
\mathcal{O}(c_1'-1,c_1''-1)\oplus 
\mathcal{O}(1,0)\oplus \mathcal{O}(0,1)\oplus \mathcal{O}^{\oplus r-2}
\to \mathcal{E}\to 0$;
\item[$(3)$] $0\to \mathcal{O}(-1,-1)\to 
\mathcal{O}(c_1'-1,c_1''-1)\oplus 
\mathcal{O}^{\oplus r}\to \mathcal{E}\to 0$.
\end{enumerate}
\end{thm}
Proposition~\ref{HalfMaximal} is a generalization of \cite[Lemma 1.2]{MR4052950},
and 
Theorem~\ref{nearToMaximal} is one of 
\cite[Theorem~7.1]{MR4453350}.
Proposition~\ref{HalfMaximal} and Theorem~\ref{nearToMaximal} will be 
used in the proof
of Theorem~\ref{Chern(2,2)}.

The content of this paper is as follows.
In Section~\ref{Preliminaries}, we briefly recall 
Bondal's theorem \cite[Theorem 6.2]{MR992977}
and its related notions
and results
used in Section~\ref{prooflatter}.
In particular we recall some finite-dimensional algebra $A$,
and fix symbols,
e.g., $G$ and $S_i$, 
related to $A$ and 
finitely generated right $A$-modules. 
In Section~\ref{lemmata},
we collect some lemmas required in Sections~\ref{neartomaximal},
\ref{proof}, and \ref{prooflatter}.
In Section~\ref{neartomaximal}, we prove Proposition~\ref{HalfMaximal} and Theorem~\ref{nearToMaximal}.
In Sections~\ref{proof} and \ref{prooflatter},
we apply the Bondal spectral sequence \cite[Theorem 1]{MR3275418}.
In Section~\ref{proof},
we prove  Theorem~\ref{Chern(2,2)} 
when $\mathcal{E}$ contains $\mathcal{O}(D)$ as a subsheaf,
where $D$ is a nonzero effective divisor.
In Section~\ref{prooflatter},
we prove Theorem~\ref{Chern(2,2)} in the remaining case.
More precisely,
in Subsection~\ref{h^1=0}, we first complete the classification 
assuming that 
$h^1(\mathcal{E})=0$,
and, in Subsection~\ref{h^1>0}, we reduce 
the case 
$h^1(\mathcal{E})>0$ 
to 
the 
case $h^1(\mathcal{E})=0$ 
by extending $H^1(\mathcal{E})\otimes \mathcal{O}$ by $\mathcal{E}$.

\subsection{Notation and conventions}\label{convention}
Throughout this paper,
we work over an algebraically closed field $K$ of characteristic zero.
Basically we follow the standard notation and terminology in algebraic
geometry. 
We denote 
by $\mathbb{Q}^2$ a smooth quadric surface over $K$,
by 
\begin{itemize}
\item $L_1$ a fiber $\mathbb{P}^1\times \{\textrm{pt}\}$ of the second projection
$\pi_2:\mathbb{P}^1\times \mathbb{P}^1\to \mathbb{P}^1$,
\end{itemize}
and by 
\begin{itemize}
\item $L_2$ a fiber $\{\textrm{pt}\}\times \mathbb{P}^1$ of the first projection
$\pi_1:\mathbb{P}^1\times \mathbb{P}^1\to \mathbb{P}^1$.
\end{itemize}
For a coherent sheaf $\mathcal{F}$,
we denote by $c_i(\mathcal{F})$ the $i$-th Chern class of $\mathcal{F}$
and by $\mathcal{F}^{\vee\vee}$ the double dual of $\mathcal{F}$.
In particular, 
\begin{itemize}
\item $c_i$ stands for $c_i(\mathcal{E})$
of the nef vector bundle $\mathcal{E}$ we are dealing with.
\end{itemize}
Since $\Pic \mathbb{Q}^2\cong \mathbb{Z}^{\oplus 2}$, we identify 
the first Chern class 
$c_1$ with 
its corresponding pair of integers:
$c_1=(c_1',c_1'')$.
For any closed subscheme $Z$
in $\mathbb{Q}^2$, 
$\mathcal{I}_Z$
denotes
the ideal sheaf of $Z$ in $\mathbb{Q}^2$;
for a point $p\in \mathbb{Q}^2$, $\mathcal{I}_p$ denotes
the ideal sheaf of 
$p\in \mathbb{Q}^2$
and $k(p)$ denotes the residue field of $p\in \mathbb{Q}^2$.
Finally we refer to \cite{MR2095472} for the definition
and basic properties of nef vector bundles.

\section{
Preliminaries
}\label{Preliminaries}
Throughout this paper, 
$G_0$, $G_1$, $G_2$, $G_3$ denote
$\mathcal{O}$, $\mathcal{O}(1,0)$, $\mathcal{O}(0,1)$, $\mathcal{O}(1,1)$,
respectively,
on $\mathbb{Q}^2$.
An important property of the 
collection 
$(G_0,G_1,G_2,G_3)$
is that it is a full strong exceptional collection in 
$D^b(\mathbb{Q}^2)$,
where $D^b(\mathbb{Q}^2)$ denotes the bounded derived category
of (the abelian category of) coherent sheaves on $\mathbb{Q}^2$.
Here we use the term ``collection'' to mean ``family'', not ``set''.
Thus an exceptional collection is also called an exceptional sequence. 
We refer to \cite{MR2244106} for the definition of a full strong exceptional sequence.

Denote by $G$ the direct sum $\bigoplus_{i=0}^3G_i$ of $G_0$, $G_1$, $G_2$, and $G_3$,
and by $A$ the endomorphism ring $\End(G)$ of $G$.
Then $A$ is a finite-dimensional $K$-algebra,
and $G$ is a left $A$-module.
Note that $\Ext^q(G, \mathcal{F})$ is a finitely generated 
right $A$-module for a coherent sheaf $\mathcal{F}$ 
on $\mathbb{Q}^2$.
We denote by $\module A$ the category of finitely generated right $A$-modules
and by $D^b(\module A)$ the bounded derived category
of $\module A$.
Let $p_i:G\to G_i$ be the projection,
and $\iota_i:G_i\hookrightarrow G$ the inclusion.
Set $e_i=\iota_i\circ p_i$ in $A$.
Any finitely generated right $A$-module $V$
has an ascending filtration 
\[0=V^{\leq -1}\subset V^{\leq 0}\subset V^{\leq 1}\subset 
V^{\leq 2}
\subset V^{\leq 3}=V\]
by right $A$-submodules,
where $V^{\leq i}$ is defined to be $\bigoplus_{j\leq i}Ve_j$.
Set $\Gr^iV=V^{\leq i}/V^{\leq i-1}$.
Let $S_i$ denote the right $A$-module 
such that $\Gr^iS_i=K$
and $\Gr^jS_i=0$ for any $j\neq i$.

It follows from Bondal's theorem \cite[Theorem 6.2]{MR992977}
that 
\[\RHom(G,\bullet):D^b(\mathbb{Q}^2)\to D^b (\module A)\]
is an exact equivalence,
and its quasi-inverse is 
\[\bullet\lotimes_AG:D^b(\module A)\to D^b(\mathbb{Q}^2).\]
Since 
$\RHom(G,\mathcal{O}(-1,-1)[2])\cong S_3$,
we obtain
\[S_3\lotimes_AG\cong \mathcal{O}(-1,-1)[2].\]
Since 
$\RHom(G,\mathcal{O}(0,-1)[1])\cong S_2$,
we have 
\[S_2\lotimes_AG\cong \mathcal{O}(0,-1)[1].\]
Similarly we have 
\[S_1\lotimes_AG\cong \mathcal{O}(-1,0)[1]\textrm{ and }
S_0\lotimes_AG\cong \mathcal{O}.\]
We will use these facts 
freely 
in this paper.
We will also apply the Bondal spectral sequence \cite[Theorem 1]{MR3275418}
to a coherent sheaf $\mathcal{F}$ on $\mathbb{Q}^2$
\begin{equation}\label{BondalSpectral}
E_2^{p,q}=\caltor_{-p}^A(\Ext^q(G,\mathcal{F}),G)
\Rightarrow
E^{p+q}=
\begin{cases}
\mathcal{F}& \textrm{if}\quad  p+q= 0\\
0& \textrm{if}\quad  p+q\neq 0.
\end{cases}
\end{equation}

Recall the following theorem~\cite[Theorem~0.1]{MR4052950},
which will be applied in Section~\ref{proof}.
\begin{thm}\label{Chern(2,1)}
Let $\mathcal{E}$ be a nef vector bundle of rank $r$
on a smooth quadric surface $\mathbb{Q}^2$ over an algebraically closed field $K$ 
of characteristic zero.
If $\det\mathcal{E}\cong \mathcal{O}(2,1)$,
then 
$\mathcal{E}$ is isomorphic to one of the following vector bundles
or fits in one of the following exact sequences:
\begin{enumerate}
\item[$(1)$]
$\mathcal{O}(2,1)\oplus \mathcal{O}^{\oplus r-1}$;
\item[$(2)$] 
$\mathcal{O}(1,1)\oplus
\mathcal{O}(1,0)\oplus  \mathcal{O}^{\oplus r-2}$;
\item[$(3)$] 
$0\to \mathcal{O}
\to 
\mathcal{O}(1,0)^{\oplus 2}\oplus  
\mathcal{O}(0,1)\oplus  
\mathcal{O}^{\oplus r-2}\to \mathcal{E}\to 0$;\\
\item[$(4)$] $0\to \mathcal{O}(-1,-1)
\oplus \mathcal{O}(-1,0)
\to \mathcal{O}^{\oplus r+2}\to \mathcal{E}\to 0$;
\item[$(5)$] $0\to \mathcal{O}(-2,-1)
\to \mathcal{O}^{\oplus r+1}\to \mathcal{E}\to 0$.
\end{enumerate}
\end{thm}

\section{Lemmas}\label{lemmata}
\begin{lemma}\label{useful}
Let $\mathcal{E}$ be a nef vector bundle on a smooth complete surface $S$,
and $\mathcal{L}$ a line bundle on $S$.
Suppose that there exists a non-zero section 
$s\in H^0(\mathcal{E}\otimes\mathcal{L}^{-1})$ whose zero locus $(s)_0$ 
has dimension $\leq 0$.
Let $\mathcal{F}$ be the cokernel of the injection 
$\mathcal{L}\hookrightarrow \mathcal{E}$
induced by $s$,
$\mathcal{F}^{\vee\vee}$ the double dual of $\mathcal{F}$,
and $\mathcal{Q}$  the cokernel of the natural morphism 
$\mathcal{F}\to \mathcal{F}^{\vee\vee}$.
Then 
we have 
\[
0\leq c_2(\mathcal{F}^{\vee\vee})+\length \mathcal{Q}=c_2(\mathcal{F})\leq 
c_1(\mathcal{F})^2.
\]
\end{lemma}
\begin{proof}
First note that $\mathcal{F}$ is torsion-free
by \cite[Lemma~5.4]{MR4453350}, since $\dim (s)_0\leq 0$.
Hence the natural morphism 
$\mathcal{F}\to \mathcal{F}^{\vee\vee}$ is injective.
The exact sequence $0\to \mathcal{L}\to \mathcal{E}\to \mathcal{F}\to 0$
shows that 
\begin{equation}\label{c2fandc2}
c_2(\mathcal{F})+c_1(\mathcal{L})c_1(\mathcal{F})=c_2.
\end{equation}
Note here that  $Q$ is a zero-dimensional coherent sheaf
and that
\[\length Q=-c_2(\mathcal{Q}).\]
The 
exact sequence
$0\to \mathcal{F}\to \mathcal{F}^{\vee\vee}\to \mathcal{Q}\to 0$
then 
implies that 
\[c_2(\mathcal{F})=c_2(\mathcal{F}^{\vee\vee})+\length \mathcal{Q}.\]
Now recall that 
\[c_2\leq c_1(c_1-c_1(\mathcal{L}))\] 
by \cite[Lemma~10.1]{MR4453350}.
Since $c_1(\mathcal{F})=c_1-c_1(\mathcal{L})$,
this inequality together with (\ref{c2fandc2})
shows that $c_2(\mathcal{F})\leq 
c_1(\mathcal{F})^2$.
Finally note that $c_2(\mathcal{F}^{\vee\vee})\geq 0$,
since $\mathcal{F}^{\vee\vee}$ is a nef vector bundle
by \cite[Lemma 9.1]{MR4148799}.
\end{proof}

\begin{lemma}\label{length2Z}
Let $Z$ be a zero-dimensional closed subscheme of length two on $\mathbb{Q}^2$,
and suppose that $\length Z\cap L_1\leq 1$ and $\length Z\cap L_2\leq 1$.
Then $\mathcal{I}_Z$ fits in an exact sequence
\[0\to \mathcal{O}(-2,-2)\to \mathcal{O}(-1,-1)^{\oplus 2}\to \mathcal{I}_Z\to 0.\]
\end{lemma}
\begin{proof}
The assumption 
that $\length Z\cap L_1\leq 1$ and $\length Z\cap L_2\leq 1$
shows that $Z$ is a complete intersection
of two irreducible 
divisors of type $(1,1)$.
Therefore we obtain the exact sequence above.
\end{proof}

\begin{lemma}\label{(1,0)sum(0,1)}
Let $\mathcal{G}$ and  $\mathcal{Q}$ be coherent sheaves on $\mathbb{Q}^2$,
and consider an exact sequence
\[0\to \mathcal{G}\to \mathcal{O}(a,0)^{\oplus e}\oplus\mathcal{O}(0,b)^{\oplus f}
\oplus \mathcal{O}^{\oplus r} 
\to \mathcal{Q}\to 0,\]
where $a$ and $b$ are positive integers, 
and $e$, $f$, $r$ are non-negative integers.
Suppose 
that $\mathcal{Q}$ is zero-dimensional
and 
that one of the following holds:
\begin{enumerate}
\item[$(1)$] $e=0$ or $f=0$;
\item[$(2)$] $e>0$, $f>0$, and $a=b=1$.
\end{enumerate}
Then $\mathcal{G}$ admits a negative degree quotient on a line.
\end{lemma}
\begin{proof}
We may assume that $\mathcal{Q}$ is the residue field $k(p)$ of a point 
$p\in \mathbb{Q}^2$. Indeed, there exists a surjection
$\mathcal{Q}\to k(p)$ for some $p\in \mathbb{Q}^2$,
and let $\mathcal{K}$ be the kernel of the composite of 
$\mathcal{O}(a,0)^{\oplus e}\oplus\mathcal{O}(0,b)^{\oplus f}
\oplus \mathcal{O}^{\oplus r} 
\to \mathcal{Q}$ and $\mathcal{Q}\to k(p)$.
Then $\mathcal{G}$ is a subsheaf of $\mathcal{K}$,
and 
$\mathcal{K}/\mathcal{G}$ has dimension $\leq 0$.
Hence $\mathcal{G}$\textbar$_{L_i}$ 
admits a negative degree
quotient,
if $\mathcal{K}$\textbar$_{L_i}$ admits a negative degree quotient for
some line $L_i$ passing through $p$.

Suppose that $\mathcal{Q}=k(p)$. If the composite 
\[\mathcal{O}(a,0)^{\oplus e}\oplus\mathcal{O}(0,b)^{\oplus f}
\hookrightarrow
\mathcal{O}(a,0)^{\oplus e}\oplus\mathcal{O}(0,b)^{\oplus f}
\oplus \mathcal{O}^{\oplus r}
\to 
k(p)
\]
is zero, then we have a surjection $\mathcal{O}^{\oplus r}\to k(p)$,
and its kernel is isomorphic to $\mathcal{I}_p\oplus \mathcal{O}^{\oplus r-1}$
and thus $\mathcal{G}$ admits $\mathcal{I}_p\oplus \mathcal{O}^{\oplus r-1}$
as a quotient. Hence $\mathcal{G}$\textbar$_{L_i}$ admits a negative degree quotient.
 If the composite 
\[\mathcal{O}(a,0)^{\oplus e}\oplus\mathcal{O}(0,b)^{\oplus f}
\hookrightarrow
\mathcal{O}(a,0)^{\oplus e}\oplus\mathcal{O}(0,b)^{\oplus f}
\oplus \mathcal{O}^{\oplus r}
\to 
k(p)
\]
is surjective, let $\mathcal{K}$ be its kernel.
Then $h^1(\mathcal{K})=0$
and $\mathcal{G}$ is
isomorphic to $\mathcal{K}\oplus \mathcal{O}^{\oplus r}$.
Thus 
the problem is reduced to the case $r=0$,
and 
we assume that $r=0$.
Then we may assume that some direct summand, say, $\mathcal{O}(a,0)$
yields a surjection $\mathcal{O}(a,0)\to k(p)$.
If $f=0$, then $\mathcal{G}$ is isomorphic to $\mathcal{I}_p(a,0)\oplus 
\mathcal{O}(a,0)^{\oplus e-1}$,
and thus $\mathcal{G}\mid_{L_2}$ admits $\mathcal{O}_{L_2}(-1)$ as a quotient
for the line $L_2$ passing through $p$.
If $f>0$ and $a=b=1$, then 
$\mathcal{G}$ is isomorphic to $\mathcal{I}_p(1,0)\oplus 
\mathcal{O}(1,0)^{\oplus e-1}\oplus \mathcal{O}(0,1)^{\oplus f}$,
since 
$\Ext^1(\mathcal{O}(1,0)^{\oplus e-1}\oplus 
\mathcal{O}(0,1)^{\oplus f},\mathcal{O}(1,0))=0$.
Therefore $\mathcal{G}$\textbar$_{L_2}$ admits $\mathcal{O}_{L_2}(-1)$ as a quotient.
\end{proof}

\begin{lemma}\label{R-R}
Let $\mathcal{E}$ be a 
vector bundle of rank $r$ on $\mathbb{Q}^2$
with $c_1=(2,2)$.
Then 
\[\chi(\mathcal{E}(p,q))=8-c_2+2p+2q+r(p+1)(q+1)\]
for any integers $p$, $q$.
In particular, $\chi(\mathcal{E}(-1,0))=\chi(\mathcal{E}(0,-1))=6-c_2$.
\end{lemma}
\begin{proof}
The Riemann-Roch formula for a vector bundle $\mathcal{E}(p,q)$ of rank $r$ on $\mathbb{Q}^2$
with $c_1(\mathcal{E})=(c_1', c_1'')$
is
\[\chi(\mathcal{E}(p,q))=c_1'c_1''-c_2+(q+1)c_1'+(p+1)c_1''+r(p+1)(q+1).\]
Since $(c_1', c_1'')=(2,2)$ by assumption, 
this 
yields the desired formula.
\end{proof}

\begin{lemma}\label{Ext^2vanishing}
Let $\mathcal{E}$ be a nef vector bundle of rank $r$ on $\mathbb{Q}^2$
with $c_1=(2,2)$.
Then 
\[\Ext^2(G,\mathcal{E})=0.\]
\end{lemma}
\begin{proof}
It follows from \cite[Lemma~4.5. (1)]{MR4453350} that 
$H^q(\mathcal{E}(1,1))=0$ for $q>0$.
Note here that we have the following exact sequences for integers $a$, $b$:
\[
\begin{split}
H^1(\mathcal{E}\vert_{L_1}(a))
\to H^2(\mathcal{E}(a, b-1))\to H^2(\mathcal{E}(a,b));
\\
H^1(\mathcal{E}\vert_{L_2}(b))\to H^2(\mathcal{E}(a-1, b))\to H^2(\mathcal{E}(a,b)).
\end{split}
\]
Since $\mathcal{E}$ is nef, both $H^1(\mathcal{E}\vert_{L_1}(a))$
and $H^1(\mathcal{E}\vert_{L_2}(b))$ vanish if $a, b\geq -1$.
Hence we see that $\Ext^2(G,\mathcal{E})=0$.
\end{proof}

\begin{lemma}\label{Ext^1}
Let $\mathcal{E}$ be a nef vector bundle of rank $r$ on $\mathbb{Q}^2$
with $c_1=(2,2)$.
Suppose 
that $H^1(\mathcal{E})=0$ and 
that 
$\Hom(\mathcal{O}(1,0),\mathcal{E})=\Hom(\mathcal{O}(0,1),\mathcal{E})=0$.
Then $c_2\geq 6$, and 
$\Ext^1(G,\mathcal{E})$
fits in 
the following exact sequence of right $A$-modules:
\[0\to S_1^{\oplus c_2-6}\oplus S_2^{\oplus c_2-6}\to 
\Ext^1(G,\mathcal{E})
\to S_3^{\oplus c_2-4}\to 0.\]
Moreover we have an isomorphism of right $A$-modules:
\[\Hom(G,\mathcal{E})\cong S_0^{\oplus r+8-c_2}.\]
\end{lemma}
\begin{proof}
Since both 
$\Hom(\mathcal{O}(1,0),\mathcal{E})$
and $\Hom(\mathcal{O}(0,1),\mathcal{E})$ vanish by assumption,
Lemmas~\ref{R-R} and \ref{Ext^2vanishing}
imply
that 
\[\begin{split}
-h^1(\mathcal{E}(-1,0))&
=\chi(\mathcal{E}(-1,0))=6-c_2;
\\
-h^1(\mathcal{E}(0,-1))&
=\chi(\mathcal{E}(0,-1))=6-c_2;
\\
-h^1(\mathcal{E}(-1,-1))&
=\chi(\mathcal{E}(-1,-1))=4-c_2.
\end{split}
\]
Thus $c_2-6=h^1(\mathcal{E}(-1,0))\geq 0$. Moreover 
the above formulas 
together with 
the assumption $h^1(\mathcal{E})=0$
show the structure of the right $A$-module $\Ext^1(G,\mathcal{E})$,
i.e., the desired 
exact sequence 
of right $A$-modules.
Furthermore we see by Lemmas~\ref{R-R}, \ref{Ext^2vanishing},
and the assumption $h^1(\mathcal{E})=0$ 
that $h^0(\mathcal{E})=\chi(\mathcal{E})=r+8-c_2$,
and thus we obtain $\Hom (G,\mathcal{E})\cong S_0^{\oplus r+8-c_2}$.
\end{proof}

\begin{lemma}\label{h^1EvanishingCondition}
Let $\mathcal{E}$ be a nef vector bundle of rank $r$ on $\mathbb{Q}^2$
with $c_1=(2,2)$. 
If $c_2\leq 7$, then $H^1(\mathcal{E})=0$. 
\end{lemma}
\begin{proof}
This follows from \cite[Lemma~4.5. (2)]{MR4453350}.
\end{proof}

\section{Set-up for 
Sections~\ref{neartomaximal} and \ref{proof}
}\label{Set-up}
Let $\mathcal{E}$ be a nef vector bundle of rank $r$ on $\mathbb{Q}^2$
with 
$c_1=(c_1',c_1'')$.
Suppose that 
\[\Hom (\mathcal{O}(a+1,b),\mathcal{E})
=\Hom (\mathcal{O}(a,b+1),\mathcal{E})=0
\textrm{ and }\Hom (\mathcal{O}(a,b),\mathcal{E})\neq 0\]
for some integers $a$, $b$.
We have an inclusion $\mathcal{O}(a,b)\hookrightarrow \mathcal{E}$,
and let $\mathcal{F}$ be its cokernel.
Set $\mathcal{L}:=\mathcal{O}(a,b)$, and we will apply Lemma~\ref{useful}.
Let $\mathcal{Q}$ be the cokernel of the natural inclusion 
$\mathcal{F}\hookrightarrow \mathcal{F}^{\vee\vee}$.
Then we have 
\begin{equation}\label{interesting}
0\leq c_2(\mathcal{F}^{\vee\vee})+\length \mathcal{Q}=c_2(\mathcal{F})\leq 
2(c_1'-a)(c_1''-b).
\end{equation}
Throughout 
Sections~\ref{neartomaximal} and \ref{proof},
we will also use the symbols $\mathcal{F}$ and $\mathcal{Q}$.

\section{Nef vector bundles having a near-to-maximal degree line bundle}\label{neartomaximal}

\begin{prop}[=Proposition~\ref{HalfMaximal}]
Let $\mathcal{E}$ be a nef vector bundle of rank $r$ on 
a smooth quadric surface $\mathbb{Q}^2$
with first Chern class $c_1=(c_1',c_1'')$.
Suppose that $\Hom (\mathcal{O}(c_1',b),\mathcal{E})\neq 0$
and $\Hom (\mathcal{O}(c_1',b+1),\mathcal{E})=0$ for some integer $b\leq c_1''$.
Then $\mathcal{E}$ fits in the following exact sequence:
\[0\to \mathcal{O}(c_1',b)\to \mathcal{E}\to \pi_2^*\mathcal{F}_2\to 0,\]
where $\pi_2:\mathcal{Q}\to \mathbb{P}^1$ denotes the second projection
and $\mathcal{F}_2$ is a nef vector bundle on $\mathbb{P}^1$ with degree $c_1''-b$. 
In particular, if $b=c_1''$, then 
$\mathcal{E}\cong \mathcal{O}(c_1',c_1'')\oplus \mathcal{O}^{\oplus r-1}$,
and 
if $b\leq c_1''-1$, then $\mathcal{E}$ fits in the following exact sequence:
\[0\to \mathcal{O}^{\oplus c_1''-b-1}\to 
\mathcal{O}(c_1',b)\oplus \mathcal{O}(0,1)^{\oplus c_1''-b}
\oplus \mathcal{O}^{\oplus r-2}
\to \mathcal{E}\to 0.\]
\end{prop}

\begin{proof}
Suppose that $\Hom (\mathcal{O}(c_1',b),\mathcal{E})\neq 0$.
Since $\mathcal{E}$ is nef, we have 
$\Hom (\mathcal{O}(c_1'+1,b),\mathcal{E})=0$.
Moreover $\Hom (\mathcal{O}(c_1',b+1),\mathcal{E})=0$ by assumption.
It then follows from (\ref{interesting}) that
that $\mathcal{Q}=0$.
Hence $\mathcal{F}\cong \mathcal{F}^{\vee\vee}$,
and $\mathcal{F}$ is a nef vector bundle with first Chern class $(0,c_1''-b)$.
Thus $\mathcal{F}\cong \pi_2^*\mathcal{F}_2$
for some nef vector bundle $\mathcal{F}_2$ on $\mathbb{P}^1$ of
degree $c_1''-b$. 
If $b=c_1''$, then 
$\mathcal{E}\cong \mathcal{O}(c_1',c_1'')\oplus \mathcal{O}^{\oplus r-1}$.
If $b\leq c_1''-1$, we obtain the result
by \cite[Lemma 1.1]{MR4052950}.
This completes the proof of Proposition~\ref{HalfMaximal}.
\end{proof}

\begin{thm}[=Theorem~\ref{nearToMaximal}]
Let $\mathcal{E}$ be a nef vector bundle of rank $r$
on a smooth quadric surface $\mathbb{Q}^2$
with first Chern class $
(c_1',c_1'')$.
Suppose that 
$\Hom (\mathcal{O}(c_1'-1,c_1''),\mathcal{E})
=\Hom (\mathcal{O}(c_1',c_1''-1),\mathcal{E})=0$
and that $\Hom (\mathcal{O}(c_1'-1,c_1''-1),\mathcal{E})\neq 0$.
Then $\mathcal{E}$ satisfies one of the following:
\begin{enumerate}
\item[$(1)$] $\mathcal{E}\cong \mathcal{O}(c_1'-1,c_1''-1)\oplus\mathcal{O}(1,1)\oplus
\mathcal{O}^{\oplus r-2}$;
\item[$(2)$]
$0\to \mathcal{O}\to 
\mathcal{O}(c_1'-1,c_1''-1)\oplus 
\mathcal{O}(1,0)\oplus \mathcal{O}(0,1)\oplus \mathcal{O}^{\oplus r-2}
\to \mathcal{E}\to 0$;
\item[$(3)$] $0\to \mathcal{O}(-1,-1)\to 
\mathcal{O}(c_1'-1,c_1''-1)\oplus 
\mathcal{O}^{\oplus r}\to \mathcal{E}\to 0$.
\end{enumerate}
\end{thm}

\begin{proof}
First we have 
$0\leq c_2(\mathcal{F}^{\vee\vee})+\length \mathcal{Q}=c_2(\mathcal{F})\leq 2$
by  (\ref{interesting}).

For the case
$c_2(\mathcal{F})=0$,
we have
$\length \mathcal{Q}=0$ and $\mathcal{F}\cong \mathcal{F}^{\vee\vee}$.
Moreover $\mathcal{F}\cong \mathcal{O}(1,1)\oplus\mathcal{O}^{\oplus r-2}$
by \cite[Theorem 7.1]{MR4453350} since $c_2(\mathcal{F})=0$.
Hence $\mathcal{E}\cong \mathcal{O}(c_1'-1,c_1''-1)\oplus\mathcal{O}(1,1)\oplus
\mathcal{O}^{\oplus r-2}$.
This is Case (1) of Theorem~\ref{nearToMaximal}.

For the case 
$c_2(\mathcal{F})=1$,
if $\length \mathcal{Q}=0$, then $\mathcal{F}\cong \mathcal{F}^{\vee\vee}$,
and $\mathcal{F}\cong
\mathcal{O}(1,0)\oplus \mathcal{O}(0,1)\oplus\mathcal{O}^{\oplus r-3}$
by \cite[Theorem~7.1]{MR4453350}.
If $\length \mathcal{Q}=1$, then $\mathcal{Q}\cong k(p)$ for 
some point $p\in \mathbb{Q}^2$,
and $c_2(\mathcal{F}^{\vee\vee})=0$.
Thus we infer that 
$\mathcal{F}^{\vee\vee}\cong \mathcal{O}(1,1)\oplus\mathcal{O}^{\oplus r-2}$
by \cite[Theorem 7.1]{MR4453350}.
Hence $\mathcal{F}\cong 
\mathcal{I}_p(1,1)\oplus\mathcal{O}^{\oplus r-2}$.
In both cases, $\mathcal{F}$ fits in the following exact sequence:
\[0\to \mathcal{O}\to \mathcal{O}(1,0)\oplus \mathcal{O}(0,1)\oplus \mathcal{O}^{r-2}
\to \mathcal{F}\to 0.\]
Therefore $\mathcal{E}$ belongs to Case (2) of Theorem~\ref{nearToMaximal}.

Suppose that 
$c_2(\mathcal{F})=2$.
If $\length \mathcal{Q}=0$, then $\mathcal{F}\cong \mathcal{F}^{\vee\vee}$,
and $\mathcal{F}$ fits in the following exact sequence:
\[0\to \mathcal{O}(-1,-1)\to \mathcal{O}^{\oplus r}\to \mathcal{F}\to 0\]
by \cite[Theorem~7.1]{MR4453350}.
Hence $\mathcal{E}$ belongs to Case (3) of Theorem~\ref{nearToMaximal}.
If $\length \mathcal{Q}=1$, then $c_2(\mathcal{F}^{\vee\vee})=1$,
and thus
$\mathcal{F}^{\vee\vee}\cong 
\mathcal{O}(1,0)\oplus \mathcal{O}(0,1)\oplus\mathcal{O}^{\oplus r-3}$
by \cite[Theorem 7.1]{MR4453350}.
Lemma~\ref{(1,0)sum(0,1)}
then implies that 
$\mathcal{F}$ admits a negative degree quotient on a line.
This however contradicts that 
$\mathcal{E}$ is nef. Therefore this case does not arise.
If $\length \mathcal{Q}=2$, then $c_2(\mathcal{F}^{\vee\vee})=0$,
and thus
$\mathcal{F}^{\vee\vee}\cong 
\mathcal{O}(1,1)\oplus\mathcal{O}^{\oplus r-2}$
by \cite[Theorem 7.1]{MR4453350}.
Hence there exists a surjection $\mathcal{O}(1,1)\to \mathcal{Q}$,
since $\mathcal{F}$ cannot admit a negative degree quotient.
Therefore $\mathcal{Q}\cong \mathcal{O}_Z$ for some zero-dimensional
closed subscheme $Z$ of length two.
Moreover $\length Z\cap L_1\leq 1$ and $\length Z\cap L_2\leq 1$,
since $\mathcal{F}$ cannot admit a negative degree quotient.
Lemma~\ref{length2Z} then implies 
that $\mathcal{F}$ fits in the following exact sequence:
\[0\to \mathcal{O}(-1,-1)\to \mathcal{O}^{\oplus r}\to \mathcal{F}\to 0.\]
Hence $\mathcal{E}$ belongs to Case (3) of Theorem~\ref{nearToMaximal}.
This completes the proof of Theorem~\ref{nearToMaximal}.
\end{proof}

\section{Proof of Theorem~\ref{Chern(2,2)} when $\mathcal{E}$
contains an effective divisor}\label{proof}
In this section,
we prove  Theorem~\ref{Chern(2,2)} 
when $\Hom(\mathcal{O}(a,b),\mathcal{E})\neq 0$,
where $a$ and $b$ are non-negative integers such that $a+b\geq 1$.

Suppose that $a+b\geq 2$.
Then we obtain Cases (1), (2), (3), (4), (5), (6-1) of Theorem~\ref{Chern(2,2)}
using Proposition~\ref{HalfMaximal} and Theorem~\ref{nearToMaximal}.

In the following in this section,
we consider the case $a+b=1$.
Moreover we may assume, without loss of generality, that 
$(a,b)=(0,1)$.
Since we have completed the case $a+b\geq 2$,
we may assume that 
\begin{equation}\label{WithoutMaxSub}
\Hom(\mathcal{O}(p,q),\mathcal{E})=0
\textrm{ for }(p,q)=(2,0), (1,1), (0,2).
\end{equation}
We then apply the set-up in Section~\ref{Set-up} with $(c_1',c_1'')=(2,2)$.
We have the following exact sequence
\[0\to \mathcal{F}\to \mathcal{F}^{\vee\vee}\to \mathcal{Q}\to 0,\] 
and we see that $\mathcal{F}^{\vee\vee}$ is a nef vector bundle
of rank $r-1$ with first Chern class $(2,1)$.
It follows from \eqref{interesting} that 
\[
0\leq c_2(\mathcal{F}^{\vee\vee})+\length \mathcal{Q}=c_2(\mathcal{F})\leq 
4.
\]
Since we have an exact sequence
\[0\to \mathcal{O}(-1,0)\to \mathcal{E}(-1,-1)
\to \mathcal{F}(-1,-1)\to 0,\]
it follows from $\Hom(\mathcal{O}(1,1),\mathcal{E})=0$
that $H^0(\mathcal{F}(-1,-1))=0$.
Thus 
\begin{equation}\label{4.2.1}
h^0(\mathcal{F}^{\vee\vee}(-1,-1))\leq \length \mathcal{Q}\leq 4-c_2(\mathcal{F}^{\vee\vee}).
\end{equation}

The basic strategy for dealing with the case $\Hom(\mathcal{O}(0,1),\mathcal{E})\neq 0$
under the assumption~\eqref{WithoutMaxSub} is as follows:
first apply Theorem~\ref{Chern(2,1)} to $\mathcal{F}^{\vee\vee}$
in order to determine the structure of $\mathcal{F}^{\vee\vee}$
based on the value of $c_2(\mathcal{F}^{\vee\vee})$;
then analyze the structure of $\mathcal{F}$ 
based on the length of $\mathcal{Q}$,
with range of \eqref{4.2.1}.
This will be carried out in Subsections \ref{4.2.2},
\ref{c2Fdoubledual=1}, \ref{c2Fdoubledual=2},
\ref{c2Fdoubledual=3}, and \ref{c2Fdoubledual=4}.

\subsection{
The case $c_2(\mathcal{F}^{\vee\vee})=0$
}\label{4.2.2}
We have 
\[\mathcal{F}^{\vee\vee}\cong 
\mathcal{O}(2,1)\oplus\mathcal{O}^{\oplus r-2}\]
by 
Theorem~\ref{Chern(2,1)}.
Hence $h^0(\mathcal{F}^{\vee\vee}(-1,-1))=2$,
and thus 
$2\leq \length\mathcal{Q}
\leq 4$
by \eqref{4.2.1}.
Since $\mathcal{F}$ cannot admit negative degree quotients,
we have a surjection $\mathcal{O}(2,1)\to \mathcal{Q}$, 
and thus 
$\mathcal{Q}\cong \mathcal{O}_Z$
for some zero-dimensional closed subscheme $Z$. Note that 
\[
2\leq \length Z\leq 4.
\]
Moreover 
$\mathcal{F}$ fits in an exact sequence
\[
0\to \mathcal{I}_Z(2,1)\to \mathcal{F}\to \mathcal{O}^{\oplus r-2}\to 0,
\]
and we have 
$\length Z\cap L_1\leq 2$ and $\length Z\cap L_2\leq 1$.

\subsubsection{
Suppose that 
$\length Z=2$.}
If $\length Z\cap L_1=2$ for some line $L_1$, then $Z$ lies on the line $L_1$,
and 
$\mathcal{I}_Z$
fits in an exact sequence
\[0\to \mathcal{O}(-2+a,-1)\to \mathcal{O}(-2,0)\oplus \mathcal{O}(a,-1)
\to \mathcal{I}_Z\to 0\]
for any $a\leq 0$.  Here we take $a$ to be $-1$. Then 
$\mathcal{F}$ fits in an exact sequence
\[
0\to \mathcal{O}(-1,0)\to \mathcal{O}(0,1)\oplus \mathcal{O}(1,0)\oplus
\mathcal{O}^{\oplus r-2}\to \mathcal{F}\to 0.
\]
Therefore $\mathcal{E}$ belongs to Case $(6$-$3)$ of Theorem~\ref{Chern(2,2)}.
If $\length Z\cap L_1\leq 1$ for any line $L_1$, then Lemma~\ref{length2Z} shows that 
$\mathcal{F}$ fits in the following exact sequence
\[
0\to \mathcal{O}(0,-1)\to \mathcal{O}(1,0)^{\oplus 2}
\oplus\mathcal{O}^{\oplus r-2}\to \mathcal{F}\to 0.\]
Therefore $\mathcal{E}$ belongs to Case $(6$-$3)$ of Theorem~\ref{Chern(2,2)}.

\subsubsection{
Suppose that 
$\length Z=3$.}
Let $p$ be an associated point of $Z$;
we have an exact sequence 
\[0\to k(p)\to \mathcal{O}_Z\to \mathcal{O}_{Z'}\to 0,\]
where $Z'$ is a zero-dimensional closed subscheme of $\mathbb{Q}^2$
of length $2$. This exact sequence induces an exact sequence
\[0\to \mathcal{I}_Z\to \mathcal{I}_{Z'}\to k(p)\to 0.\]
By taking 
$p\in Z$ suitably, 
we may assume that 
$\length Z'\cap L_i\leq 1$ for $i=1,2$.
Lemma~\ref{length2Z} then implies 
that $h^0(\mathcal{I}_{Z'}(1,1))=2$
and that $\mathcal{I}_{Z'}(1,1)$ is globally generated.
Hence we have a surjection $H^0(\mathcal{I}_{Z'}(1,1))\to H^0(k(p))$,
and thus $h^0(\mathcal{I}_Z(1,1))=1$.
Let $D_1$ be the divisor of type $(1,1)$ containing $Z$.
Since $h^0(\mathcal{O}_Z)=3$,
we see that $h^0(\mathcal{I}_Z(2,1))\geq 
h^0(\mathcal{O}(2,1))-h^0(\mathcal{O}_Z)=
3$.
Note here that reducible 
divisors of type $(2,1)$
containing $Z$
correspond to a 
linear subspace
of $H^0(\mathcal{I}_Z(2,1))$ of dimension 
two.
Therefore  $Z$ is contained in an irreducible divisor 
$D_2$ of type $(2,1)$.
Then $\dim D_1\cap D_2=0$
and $Z\subset D_1\cap D_2$.
Note that the irreducibility of $D_2$ implies 
its smoothness since its arithmetic genus is zero. 
Since the intersection number $D_1D_2=3$,
we have $\length Z=3=\length D_1\cap D_2$.
Hence $Z=D_1\cap D_2$.
Therefore $\mathcal{I}_Z$
fits in the following exact sequence:
\begin{equation}\label{atodemotukau}
0\to \mathcal{O}(-3,-2)\to 
\mathcal{O}(-2,-1)\oplus
\mathcal{O}(-1,-1)\to \mathcal{I}_Z\to 0.
\end{equation}
Hence 
$\mathcal{F}$ fits in the following exact sequence:
\[
0\to \mathcal{O}(-1,-1)\to \mathcal{O}(1,0)
\oplus\mathcal{O}^{\oplus r-1}\to \mathcal{F}\to 0,\]
and thus $\mathcal{E}$ fits in the following exact sequence:
\[0\to \mathcal{O}(-1,-1)
\to \mathcal{O}(1,0)\oplus  \mathcal{O}(0,1)\oplus  \mathcal{O}^{\oplus r-1}
\to \mathcal{E}\to 0.\]
Therefore $\mathcal{E}$ belongs to Case $(7)$ of Theorem~\ref{Chern(2,2)}.

\subsubsection{
Suppose that $\length Z=4$.}
Then $h^0(\mathcal{O}_Z)=4$, and thus $h^0(\mathcal{I}_Z(2,1))\geq 2$.
Note here that every closed subscheme $Z_1\subset Z$ of length $3$
determines uniquely an effective divisor $D_1$ containing $Z_1$ of type $(1,1)$
since $\length Z\cap L_1\leq 2$ and $\length Z\cap L_2\leq 1$.
Moreover $D_1$ does not contain $Z$;
otherwise 
$\mathcal{I}_{Z}(2,1)\vert_{D_1}$ 
and thus 
$\mathcal{F}\vert_{D_1}$ 
would admit a negative degree quotient,
which is a contradiction. 
Hence we see that the number of reducible 
divisors containing $Z$ of type $(2,1)$
is finite.
Therefore we can take two irreducible divisors $D_1$ and $D_2$ 
containing $Z$ of type $(2,1)$.
Then $\dim D_1\cap D_2=0$ and $Z\subset D_1\cap D_2$.
Since the intersection number $D_1D_2=4$,
we have $\length Z=4=\length D_1\cap D_2$.
Thus $Z=D_1\cap D_2$, and we have an exact sequence 
\[
0\to \mathcal{O}(-2,-1)\to \mathcal{O}^{\oplus 2}\to \mathcal{I}_Z(2,1)\to 0.
\]
Hence 
$\mathcal{F}$ fits in the following exact sequence:
\[
0\to \mathcal{O}(-2,-1)\to 
\mathcal{O}^{\oplus r}\to \mathcal{F}\to 0.\]
Therefore $\mathcal{E}$ belongs to Case $(8)$ of Theorem~\ref{Chern(2,2)}.

\subsection{
The case $c_2(\mathcal{F}^{\vee\vee})=1$
}\label{c2Fdoubledual=1}
We have 
$\mathcal{F}^{\vee\vee}
\cong \mathcal{O}(1,1)\oplus \mathcal{O}(1,0)\oplus \mathcal{O}^{r-3}$
by  Theorem~\ref{Chern(2,1)}.
Hence $h^0(\mathcal{F}^{\vee\vee}(-1,-1))=1$,
and thus 
$1\leq \length\mathcal{Q}\leq 3$ by \eqref{4.2.1}.
Since $\mathcal{F}$ cannot admit negative degree quotients,
we see that the composite 
\[
\mathcal{O}(1,1)\oplus\mathcal{O}(1,0)\hookrightarrow
\mathcal{F}^{\vee\vee}\to \mathcal{Q}
\]
is surjective.
Let $\mathcal{K}$ be the kernel of the surjection
$\mathcal{O}(1,1)\oplus\mathcal{O}(1,0)\to \mathcal{Q}$. 
Then $\mathcal{F}$ fits in an exact sequence 
\[0\to \mathcal{K}\to \mathcal{F}\to \mathcal{O}^{\oplus r-3}\to 0.\]
Let $\mathcal{Q}_1$ and $\mathcal{Q}_2$ denote
the image and the cokernel of the composite 
\[\mathcal{O}(1,1)\hookrightarrow
\mathcal{O}(1,1)\oplus\mathcal{O}(1,0)\to \mathcal{Q}\]
respectively.
Note that we have a surjection $\mathcal{O}(1,0)\to \mathcal{Q}_2$.
Hence $\mathcal{Q}_2\cong \mathcal{O}_{Z_2}$ for some closed 
subscheme of dimension $\leq 0$.
If $Z_2\neq \emptyset$, then $\mathcal{K}$ admits $\mathcal{I}_{Z_2}(1,0)$
as a quotient, and thus $\mathcal{F}\vert_{L_2}$ admits a negative degree quotient 
for some line $L_2$. This is a contradiction.
Hence $Z_2=\emptyset$, and thus the composite 
$\mathcal{O}(1,1)\hookrightarrow
\mathcal{O}(1,1)\oplus\mathcal{O}(1,0)\to \mathcal{Q}$
is surjective.
Thus $\mathcal{Q}\cong \mathcal{O}_Z$
for some zero-dimensional closed subscheme $Z$
with 
\[1\leq \length Z\leq 3,\]
and 
$\mathcal{K}$ fits in an exact sequence
\begin{equation}\label{unsplitK}
0\to \mathcal{I}_Z(1,1)\to \mathcal{K}\to \mathcal{O}(1,0)\to 0
\end{equation}
with
$\length Z\cap L_1\leq 2$ and $\length Z\cap L_2\leq 1$.

\subsubsection{Suppose that $\length Z=1$.}
Then $Z$ is a point $p$,
and 
$\mathcal{K}$
fits in the following exact sequence:
\[
0\to \mathcal{O}\to \mathcal{O}(1,0)^{\oplus 2}
\oplus\mathcal{O}(0,1)\to \mathcal{K}\to 0.
\]
Thus $\mathcal{E}$ fits in an exact sequence
\[0\to \mathcal{O}\to \mathcal{O}(1,0)^{\oplus 2}\oplus  
\mathcal{O}(0,1)^{\oplus 2}\oplus  
\mathcal{O}^{\oplus r-3}\to \mathcal{E}\to 0.\]
Since $\mathcal{E}$ is a vector bundle and 
$h^0(\mathcal{E}(-2,0))=h^0(\mathcal{E}(0,-2))=0$
by our assumption (\ref{WithoutMaxSub}),
we see that $\mathcal{E}$
belongs to Case $(6\textrm{-}2)$ of Theorem~\ref{Chern(2,2)}.

\subsubsection{
Suppose that $\length Z=2$.}
If $\length Z\cap L_1\leq 1$ for any line $L_1$, then 
Lemma~\ref{length2Z} shows that 
$\mathcal{K}$
fits in the following exact sequence:
\[
0\to \mathcal{O}(-1,-1)\to \mathcal{O}(1,0)
\oplus\mathcal{O}^{\oplus 2}\to \mathcal{K}\to 0.
\]
Therefore $\mathcal{E}$ belongs to Case $(7)$ of Theorem~\ref{Chern(2,2)}.
If $\length Z\cap L_1=2$ for some line $L_1$, then $Z$ 
lies on 
the line $L_1$,
and $\mathcal{I}_Z$ fits in the following exact sequence:
\[
0\to 
\mathcal{O}(-2,-1)\to \mathcal{O}(-2,0)\oplus \mathcal{O}(0,-1)\to \mathcal{I}_Z\to 0.
\]
Note here that the exact sequence~\eqref{unsplitK} does not split,
since 
$\mathcal{I}_Z(1,1)|_{L_1}$ admits a negative degree quotients.
We will apply  the Bondal spectral sequence~\eqref{BondalSpectral} to $\mathcal{K}$.
Note that $E_2^{p,q}$ is the $p$-th cohomology 
$\mathcal{H}^p(\Ext^q(G,\mathcal{K})\lotimes_A G)$
of the complex $\Ext^q(G,\mathcal{K})\lotimes_A G$.
Since $h^1(\mathcal{I}_{Z})=1$ and $h^q(\mathcal{I}_Z)=0$ for $q=0,2$,
we infer that $h^1(\mathcal{K}(-1,-1))=1$ and 
$h^q(\mathcal{K}(-1,-1))=0$ for $q=0,2$.
Since $h^q(\mathcal{I}_{Z}(1,0))=0$ for all $q$,
we see that $h^q(\mathcal{K}(0,-1))=0$ for all $q$.
Since $h^q(\mathcal{I}_{Z}(0,1))=1$ for $q=0,1$ and $h^2(\mathcal{I}_{Z}(0,1))=0$,
the non-splitting of  
\eqref{unsplitK} implies that 
$h^q(\mathcal{K}(-1,0))=0$ for $q=1,2$ and 
$h^0(\mathcal{K}(-1,0))=1$.
Since $h^0(\mathcal{I}_{Z}(1,1))=2$ and $h^q(\mathcal{I}_Z(1,1))=0$ for $q=1,2$,
we see that $h^0(\mathcal{K})=4$ and 
$h^q(\mathcal{K})=0$ for $q=1,2$.
Hence $\Ext^2(G,\mathcal{K})=0$, $\Ext^1(G,\mathcal{K})\cong S_3$,
and $\Hom(G,\mathcal{K})$ fits in an exact sequence of right $A$-modules
\[
0\to S_0^{\oplus 4}\to \Hom(G,\mathcal{K})\to S_1\to 0,
\]
which yields the following distinguished triangle
\[
\mathcal{O}^{\oplus 4}\to \Hom(G,\mathcal{K})\lotimes_A G\to 
\mathcal{O}(-1,0)[1]\to. 
\]
By taking 
cohomologies of the 
triangle above, we have  
the following exact sequence:
\[
0\to 
E_2^{-1,0}\to 
\mathcal{O}(-1,0)\to \mathcal{O}^{\oplus 4}\to E_2^{0,0}\to 0.
\]
We also see
that
$E_2^{-2,1}\cong \mathcal{O}(-1,-1)$
and that $E_2^{p,q}=0$ unless $(p,q)=(-2,1)$ 
$(-1,0)$ 
or $(0,0)$.
Since $E^{-1}=0$, this implies that $E_2^{-1,0}=E_{\infty}^{-1,0}=0$.
Since $\mathcal{K}$ fits in the 
following 
exact sequence:
\[0\to E_2^{-2,1}\to E_2^{0,0}\to \mathcal{K}\to 0,\]
the vanishing $\Ext^1(\mathcal{O}(-1,-1),\mathcal{O}(-1,0))=0$ 
shows that 
$\mathcal{K}$ fits in an exact sequence
\[
0\to \mathcal{O}(-1,-1)\oplus\mathcal{O}(-1,0)\to \mathcal{O}^{\oplus 4}\to 
\mathcal{K}\to 0.
\]
Thus $\mathcal{E}$ fits in the following exact sequence:
\[
0\to \mathcal{O}(-1,-1)\oplus \mathcal{O}(-1,0)
\to \mathcal{O}(0,1)\oplus 
\mathcal{O}^{\oplus r+1}\to \mathcal{E}\to 0.
\]
Therefore $\mathcal{E}$ belongs to Case $(7)$ of Theorem~\ref{Chern(2,2)}.

\subsubsection{
Suppose that $\length Z=3$.}
As we have observed  in the case $\length Z=3$ in Subsection~\ref{4.2.2},
$Z$ is the intersection $D_1\cap D_2$ of an effective divisor $D_1$ of type $(1,1)$
and an irreducible divisor $D_2$ of type $(2,1)$,
and $\mathcal{I}_Z$ fits in the exact sequence~\eqref{atodemotukau}.
Note here that the exact sequence~\eqref{unsplitK} does not split
since $\mathcal{I}_Z(1,1)\vert_{D_1}$ admits a negative degree quotient.
Again we will apply 
the Bondal spectral sequence~\eqref{BondalSpectral} to $\mathcal{K}$.
Since $h^1(\mathcal{I}_{Z})=2$ and $h^q(\mathcal{I}_Z)=0$ for $q=0,2$,
we infer that $h^1(\mathcal{K}(-1,-1))=2$ and $h^q(\mathcal{K}(-1,-1))=0$ for $q=0,2$.
Since $h^1(\mathcal{I}_{Z}(1,0))=1$ and 
$h^q(\mathcal{I}_{Z}(1,0))=0$ for $q=0,2$,
we see that $h^1(\mathcal{K}(0,-1))=1$ and $h^q(\mathcal{K}(0,-1))=0$ for $q=0,2$.
Since $h^1(\mathcal{I}_{Z}(0,1))=1$ and $h^q(\mathcal{I}_{Z}(0,1))=0$
for $q=0,2$,
the non-splitting of  
\eqref{unsplitK} implies that 
$h^q(\mathcal{K}(-1,0))=0$ for all $q$.
Since $h^0(\mathcal{I}_{Z}(1,1))=1$ and $h^q(\mathcal{I}_Z(1,1))=0$ for $q=1,2$,
we see that $h^0(\mathcal{K})=3$ and $h^q(\mathcal{K})=0$ for $q=1,2$.
Hence $\Ext^2(G,\mathcal{K})=0$, $\Hom(G,\mathcal{K})\cong S_0^{\oplus 3}$,
and $\Ext^1(G,\mathcal{K})$ fits in an exact sequence of right $A$-modules
\[
0\to S_2\to \Ext^1(G,\mathcal{K})\to S_3^{\oplus 2}\to 0,
\]
which yields the following distinguished triangle:
\[
\mathcal{O}(-1,-1)^{\oplus 2}[1]\to
\mathcal{O}(0,-1)[1]\to \Ext^1(G,\mathcal{K})\lotimes_A G\to. 
\]
Thus 
$E_2^{-2,1}$ and $E_2^{-1,1}$ lie in the following  exact sequence:
\[
0\to E_2^{-2,1}\to \mathcal{O}(-1,-1)^{\oplus 2}
\xrightarrow{\pi} \mathcal{O}(0,-1)\to E_2^{-1,1}\to 0.
\]
Moreover $E_2^{0,0}\cong \mathcal{O}^{\oplus 3}$, and 
$E_2^{p,q}=0$ unless $(p,q)=(-2,1)$, $(-1,1)$ or $(0,0)$.
Since $\mathcal{K}$ fits in the exact sequence
$0\to E_3^{0,0}\to \mathcal{K}\to E_2^{-1,1}\to 0$,
$E_2^{-1,1}\vert_{L_2}$ cannot admit a negative degree quotient,
and thus 
the restriction $\pi|_{L_2}$ of 
the morphism $\pi$ above to any line $L_2$ is surjective.
Hence 
$\pi$ itself is surjective, i.e., 
$E_2^{-1,1}=0$ and $E_2^{-2,1}\cong \mathcal{O}(-2,-1)$.
Therefore $\mathcal{K}\cong E_3^{0,0}$ and 
$\mathcal{K}$ fits in an exact sequence
\[
0\to \mathcal{O}(-2,-1)\to \mathcal{O}^{\oplus 3}
\to \mathcal{K}\to 0.
\]
Therefore $\mathcal{E}$ belongs to Case $(8)$ of Theorem~\ref{Chern(2,2)}.

\subsection{
The case
$c_2(\mathcal{F}^{\vee\vee})=2$
}\label{c2Fdoubledual=2}
We see that $\mathcal{F}^{\vee\vee}$
fits in the following exact sequence:
\[
0\to \mathcal{O}\to \mathcal{O}(1,0)^{\oplus 2}
\oplus\mathcal{O}(0,1)\oplus\mathcal{O}^{\oplus r-3}\to 
\mathcal{F}^{\vee\vee}\to 0\]
by  Theorem~\ref{Chern(2,1)}.

Suppose that $\mathcal{Q}=0$. Then $\mathcal{F}\cong \mathcal{F}^{\vee\vee}$,
and the assumption~\eqref{WithoutMaxSub} implies that
$\mathcal{E}$ belongs to Case $(6\textrm{-}2)$ of Theorem~\ref{Chern(2,2)}.

Suppose that $\mathcal{Q}\neq 0$. 
Let $\mathcal{G}$ be the kernel of the composite
\[\mathcal{O}(1,0)^{\oplus 2}
\oplus\mathcal{O}(0,1)\oplus\mathcal{O}^{\oplus r-3}\to 
\mathcal{F}^{\vee\vee}\to \mathcal{Q}\]
of the surjections.
Then $\mathcal{G}$ fits in an exact sequence
\[
0\to \mathcal{O}\to \mathcal{G}\to \mathcal{F}\to 0.
\]
Since $\mathcal{G}$ admits a negative degree quotient 
on a line by Lemma~\ref{(1,0)sum(0,1)},
this implies that $\mathcal{F}$ also admits a negative degree quotient 
on a line.
This is a contradiction.
Hence this case does not arise.

\subsection{
The case
$c_2(\mathcal{F}^{\vee\vee})=3$
}\label{c2Fdoubledual=3}
We see that
$\mathcal{F}^{\vee\vee}$
fits in the following exact sequence:
\[
0\to \mathcal{O}(-1,-1)\oplus \mathcal{O}(-1,0)\to 
\mathcal{O}^{\oplus r+1}\to \mathcal{F}^{\vee\vee}\to 0\]
by  Theorem~\ref{Chern(2,1)}.
We have $\length \mathcal{Q}=0$ or $1$.

If $\mathcal{Q}=0$, then 
$\mathcal{E}$ belongs to Case $(7)$ of Theorem~\ref{Chern(2,2)}.

If $\mathcal{Q}=k(p)$ for some point $p\in \mathbb{Q}$,
$\mathcal{F}$ fits in an exact sequence
\[
0\to \mathcal{O}(-1,-1)\oplus \mathcal{O}(-1,0)\to
\mathcal{I}_p\oplus \mathcal{O}^{\oplus r}\to \mathcal{F}\to 0.\]
Hence 
$\mathcal{F}$ fits in an exact sequence
\[
0\to \mathcal{O}(-1,-1)^{\oplus 2}\oplus \mathcal{O}(-1,0)
\xrightarrow{\varphi}
\mathcal{O}(-1,0)\oplus \mathcal{O}(0,-1)
\oplus \mathcal{O}^{\oplus r}\to \mathcal{F}\to 0.\]

We claim here that the composite of 
the inclusion $\mathcal{O}(-1,0)\to 
\mathcal{O}(-1,-1)^{\oplus 2}\oplus \mathcal{O}(-1,0)$,
the morphism $\varphi$ above 
and the projection
$\mathcal{O}(-1,0)\oplus \mathcal{O}(0,-1)
\oplus \mathcal{O}^{\oplus r}\to \mathcal{O}(-1,0)$
is non-zero.
Assume, to the contrary, that it is zero.
Then $\varphi$
induces a morphism $\mathcal{O}(-1,-1)^{\oplus 2}\to \mathcal{O}(-1,0)$,
whose cokernel is a quotient of $\mathcal{F}$.
Since $\mathcal{F}\vert_{L_1}$ cannot admit a negative degree quotient
for any line $L_1$,
the morphism $\mathcal{O}(-1,-1)^{\oplus 2}\to \mathcal{O}(-1,0)$ is surjective,
and thus 
its kernel is $\mathcal{O}(-1,-2)$.
Hence the kernel of the composite of $\varphi$
and the projection $\mathcal{O}(-1,0)\oplus \mathcal{O}(0,-1)
\oplus \mathcal{O}^{\oplus r}\to \mathcal{O}(-1,0)$
is isomorphic to $\mathcal{O}(-1,-2)\oplus \mathcal{O}(-1,0)$.
Therefore 
the exact sequence above induces 
the following exact sequence:
\[
0\to \mathcal{O}(-1,-2)\oplus \mathcal{O}(-1,0)
\xrightarrow{\psi} \mathcal{O}(0,-1)\oplus \mathcal{O}^{\oplus r}
\to \mathcal{F}\to 0.
\]
Since the composite of the inclusion 
$\mathcal{O}(-1,0)\to 
\mathcal{O}(-1,-2)\oplus \mathcal{O}(-1,0)$,
the morphism $\psi$ above and the projection
$\mathcal{O}(0,-1)\oplus \mathcal{O}^{\oplus r}
\to \mathcal{O}(0,-1)$ is zero, 
$\psi$ induces a morphism
$\mathcal{O}(-1,-2)\to \mathcal{O}(0,-1)$,
whose cokernel is a quotient of $\mathcal{F}$
and is isomorphic to either $\mathcal{O}(0,-1)$ 
or $\mathcal{O}_{D_1}(-p)$,
where $D_1$ is an effective divisor of type $(1,1)$
and $p$ is a point on $D_1$.
This is a contradiction.
Hence the claim holds.

The claim above implies that 
the composite
\[\mathcal{O}(-1,0)\to
\mathcal{O}(-1,-1)^{\oplus 2}\oplus \mathcal{O}(-1,0)
\xrightarrow{\varphi}
\mathcal{O}(-1,0)\oplus \mathcal{O}(0,-1)
\oplus \mathcal{O}^{\oplus r}\to \mathcal{O}(-1,0)\]
is an isomorphism, and thus the kernel 
of the morphism
\[\mathcal{O}(-1,-1)^{\oplus 2}\oplus \mathcal{O}(-1,0)
\xrightarrow{\varphi}
\mathcal{O}(-1,0)\oplus \mathcal{O}(0,-1)
\oplus \mathcal{O}^{\oplus r}\to \mathcal{O}(-1,0)\]
is isomorphic to 
$\mathcal{O}(-1,-1)^{\oplus 2}$.
Therefore $\mathcal{F}$ fits in an exact sequence
\[
0\to \mathcal{O}(-1,-1)^{\oplus 2}\xrightarrow{\iota}
\mathcal{O}(0,-1)
\oplus \mathcal{O}^{\oplus r}\to \mathcal{F}\to 0.\]
Since the composite of the injection $\iota$ above
and the projection 
$\mathcal{O}(0,-1)
\oplus \mathcal{O}^{\oplus r}\to
\mathcal{O}(0,-1)$ is surjective,
the exact sequence above induces the following exact sequence:
\[
0\to \mathcal{O}(-2,-1)\to
\mathcal{O}^{\oplus r}\to \mathcal{F}\to 0.\]
Hence $\mathcal{E}$ belongs to Case $(8)$ of Theorem~\ref{Chern(2,2)}.

\subsection{
The case $c_2(\mathcal{F}^{\vee\vee})=4$
}\label{c2Fdoubledual=4}
We see 
that $\mathcal{Q}=0$, that $\mathcal{F}\cong \mathcal{F}^{\vee\vee}$,
and that $\mathcal{F}$ fits in the following exact sequence:
\[
0\to \mathcal{O}(-2,-1)\to 
\mathcal{O}^{\oplus r}\to \mathcal{F}\to 0\]
by  Theorem~\ref{Chern(2,1)}.
Therefore $\mathcal{E}$ belongs to Case $(8)$ of Theorem~\ref{Chern(2,2)}.

\section{Proof of Theorem~\ref{Chern(2,2)} when $\mathcal{E}$
contains no effective divisors}\label{prooflatter}
In the following, we 
assume 
that  $\Hom (\mathcal{O}(0,1),\mathcal{E})
=\Hom (\mathcal{O}(1,0),\mathcal{E})=0$.

\subsection{The case $h^1(\mathcal{E})=0$}\label{h^1=0}
We will apply 
the Bondal spectral sequence~\eqref{BondalSpectral} to $\mathcal{E}$.
Note that $E_2^{p,q}=\mathcal{H}^p(\Ext^q(G,\mathcal{E})\lotimes_A G)$.
Lemma~\ref{Ext^2vanishing} implies that 
$E_2^{p,2}=0$,
and 
Lemma~\ref{Ext^1} shows 
that 
\[c_2\geq 6.\]

Before going into the details of the proof of Subsection~\ref{h^1=0},
we describe herein the structure of Subsection~\ref{h^1=0}
and provide a roadmap to the classification of  Theorem~\ref{Chern(2,2)}.
The proof of Subsection~\ref{h^1=0}
is divided into three parts:
in \ref{c2=6}, we will deal with the case $c_2=6$
to obtain Case $(9)$ of Theorem~\ref{Chern(2,2)};
in \ref{c2bigger6}, we will describe the common set-up for the case $c_2\geq 7$
and deal with the case $c_2=7$ to obtain Case $(11)$
of  Theorem~\ref{Chern(2,2)};
in \ref{c_2=8}, we will show that $c_2=8$
and deal with the case $c_2=8$ to obtain Cases $(12)$ and $(13)$
of Theorem~\ref{Chern(2,2)}.  

From Lemma~\ref{Ext^1}, we 
obtain a distinguished triangle
\[
\begin{split}
(S_1\lotimes_A G)^{\oplus c_2-6}\oplus (S_2\lotimes_A G)^{\oplus c_2-6}
&\to 
\Ext^1(G,\mathcal{E})\lotimes_A G\\
&\to (S_3\lotimes_A G)^{\oplus c_2-4}\to.
\end{split}
\] 
Hence we have the following distinguished triangle:
\[
\begin{split}
(\mathcal{O}(-1,0)^{\oplus c_2-6}
\oplus 
\mathcal{O}(0,-1)^{\oplus c_2-6})[1]
&\to 
\Ext^1(G,\mathcal{E})\lotimes_A G\\
&\to \mathcal{O}(-1,-1)^{\oplus c_2-4}[2]
\to.
\end{split}
\]
Taking cohomologies of the 
triangle above yields 
the following exact sequence:
\begin{equation}\label{GenSeq}
\begin{split}
0\to E_2^{-2,1}\to \mathcal{O}(-1,-1)^{\oplus c_2-4}
&\xrightarrow{\varphi}
\mathcal{O}(-1,0)^{\oplus c_2-6}
\oplus 
\mathcal{O}(0,-1)^{\oplus c_2-6}\\
&\to E_2^{-1,1}\to 0.
\end{split}
\end{equation}
We also see that $E_2^{p,1}=0$ unless $p=-2$ or $-1$.
Lemma~\ref{Ext^1} also shows
that 
\[
\Hom(G,\mathcal{E})\lotimes_A G\cong (S_0\lotimes_A G)^{\oplus r+8-c_2}
\cong
\mathcal{O}^{\oplus r+8-c_2}.\]
Taking cohomologies, we see that 
$E_2^{p,0}=0$ unless $p=0$
and that 
$E_2^{0,0}\cong \mathcal{O}^{\oplus r+8-c_2}$.
Therefore we have an exact sequence
\begin{equation}\label{E3}
0\to E_2^{-2,1}\to \mathcal{O}^{\oplus r+8-c_2}\to E_3^{0,0}\to 0,
\end{equation}
and $\mathcal{E}$ fits in an exact sequence
\begin{equation}\label{filtration}
0\to E_3^{0,0}\to \mathcal{E}\to E_2^{-1,1}\to 0.
\end{equation}
Note here that $E_2^{-1,1}$ cannot admit a negative degree quotient
since $\mathcal{E}$ is nef.

\subsubsection{
The case $c_2=6$
}\label{c2=6} 
We have 
$E_2^{-1,1}=0$ and $E_2^{-2,1}\cong \mathcal{O}(-1,-1)^{\oplus 2}$
by  (\ref{GenSeq}).
Thus $E_3^{0,0}\cong \mathcal{E}$ by (\ref{filtration})
and the exact sequence (\ref{E3}) gives 
Case $(9)$ of Theorem~\ref{Chern(2,2)}.

\subsubsection{
The case $c_2\geq 7$
} \label{c2bigger6}
Consider the composite 
\[\psi:
\mathcal{O}(-1,-1)^{\oplus c_2-4}
\xrightarrow{\varphi}
\mathcal{O}(-1,0)^{\oplus c_2-6}
\oplus 
\mathcal{O}(0,-1)^{\oplus c_2-6}
\to \mathcal{O}(0,-1)
\]
of the morphism
$\varphi$ in the exact sequence (\ref{GenSeq})
and the projection $\mathcal{O}(-1,0)^{\oplus c_2-6}
\oplus \mathcal{O}(0,-1)^{\oplus c_2-6}
\to \mathcal{O}(0,-1)$.
Since 
$E_2^{-1,1}|_{L_2}$
cannot admit a negative degree quotient
for any line $L_2$,
the composite 
$\psi$ 
is surjective, and thus  
the exact sequence~(\ref{GenSeq}) induces the following exact sequence:
\begin{equation}\label{Gen2}
\begin{split}
0\to E_2^{-2,1}
&\to 
\mathcal{O}(-2,-1)
\oplus
\mathcal{O}(-1,-1)^{\oplus c_2-6}\\
&\xrightarrow{\alpha}
\mathcal{O}(-1,0)^{\oplus c_2-6}
\oplus 
\mathcal{O}(0,-1)^{\oplus c_2-7}
\to E_2^{-1,1}\to 0.
\end{split}
\end{equation}
Note that the morphism
$\alpha$ above
is not zero 
since $E_2^{-1,1}$ cannot admit a negative degree quotient.

Now suppose that $c_2=7$. Then the exact sequence~(\ref{Gen2}) becomes
the following exact sequence:
\[
0\to E_2^{-2,1}\to 
\mathcal{O}(-2,-1)
\oplus
\mathcal{O}(-1,-1)
\xrightarrow{\alpha}
\mathcal{O}(-1,0)
\to E_2^{-1,1}\to 0.
\]
If the composite 
\[
\mathcal{O}(-1,-1)\to 
\mathcal{O}(-2,-1)
\oplus
\mathcal{O}(-1,-1)
\xrightarrow{\alpha}
\mathcal{O}(-1,0)
\]
of the inclusion $\mathcal{O}(-1,-1)\to 
\mathcal{O}(-2,-1)
\oplus
\mathcal{O}(-1,-1)$
and the morphism 
$\alpha$
is zero,
then we get an exact sequence
$0\to \mathcal{O}(-2,-1)\to \mathcal{O}(-1,0)\to E_2^{-1,1}\to 0$,
and thus $E_2^{-1,1}\cong \mathcal{O}_{D_1}(-p)$
for some effective divisor $D_1$ of type $(1,1)$ and some point $p$ on $D_1$.
This is a contradiction.
Hence the composite above is non-zero and injective,
and thus we obtain the following exact sequence:
\[
0\to E_2^{-2,1}\to 
\mathcal{O}(-2,-1)
\to
\mathcal{O}_{L_1}(-1)
\to E_2^{-1,1}\to 0.
\]
Since $E_2^{-1,1}$ cannot admit a negative degree quotient,
the morphism $\mathcal{O}(-2,-1)
\to
\mathcal{O}_{L_1}(-1)$ is non-zero,
and its image is $\mathcal{O}_{L_1}(-2)$.
Hence 
\[E_2^{-1,1}\cong k(p)\]
for some point $p\in L_1$
and $E_2^{-2,1}\cong \mathcal{O}(-2,-2)$.
The exact sequences (\ref{E3}) and (\ref{filtration}) then yield
Case $(11)$ of Theorem~\ref{Chern(2,2)}.

\subsubsection{
The case $c_2\geq 8$
}\label{c_2=8} 
Since $\mathcal{E}$ is nef, we have $c_2\leq c_1^2=8$. Hence
$c_2=8$.
The exact sequence~(\ref{Gen2})
then becomes the following exact sequence:
\[
\begin{split}
0\to E_2^{-2,1}\to 
\mathcal{O}(-2,-1)
\oplus
\mathcal{O}(-1,-1)^{\oplus 2}
&\xrightarrow{\alpha}
\mathcal{O}(-1,0)^{\oplus 2}
\oplus 
\mathcal{O}(0,-1)\\
&\to E_2^{-1,1}\to 0.
\end{split}
\]
Since $E_2^{-1,1}$ cannot admit a negative degree quotient,
the composite
\[
\mathcal{O}(-2,-1)
\oplus
\mathcal{O}(-1,-1)^{\oplus 2}
\xrightarrow{\alpha}
\mathcal{O}(-1,0)^{\oplus 2}
\oplus 
\mathcal{O}(0,-1)
\to 
\mathcal{O}(0,-1)
\]
of $\alpha$ and the projection 
$\mathcal{O}(-1,0)^{\oplus 2}
\oplus 
\mathcal{O}(0,-1)
\to 
\mathcal{O}(0,-1)$ 
is surjective.
Let $\mathcal{K}$ be its kernel.
Then $\mathcal{K}$ fits in 
the following exact sequence:
\begin{equation}
0\to E_2^{-2,1}\to 
\mathcal{K}
\xrightarrow{\mu}
\mathcal{O}(-1,0)^{\oplus 2}
\to E_2^{-1,1}\to 0.\label{K}
\end{equation}
Note that the morphism $\mu$ in (\ref{K}) 
cannot be zero since $E_2^{-1,1}$ cannot admit
a negative degree quotient.
Moreover $\mathcal{K}$ is either 
$\mathcal{O}(-3,-1)\oplus \mathcal{O}(-1,-1)$
or 
$\mathcal{O}(-2,-1)^{\oplus 2}$.
Therefore we have two cases:
\begin{enumerate}
\item[$(a)$] $\mathcal{K}\cong \mathcal{O}(-3,-1)\oplus \mathcal{O}(-1,-1);$ 
\item[$(b)$] $\mathcal{K}\cong \mathcal{O}(-2,-1)^{\oplus 2}$.
\end{enumerate}

We consider Case (a) in \ref{case(a)} 
and address Case (b) in \ref{case(b)}.

\paragraph{Case $(a)$}\label{case(a)}
If the composite 
\[\nu_a:\mathcal{O}(-1,-1)\to \mathcal{K}
\xrightarrow{\mu}
\mathcal{O}(-1,0)^{\oplus 2}\]
of the inclusion $\mathcal{O}(-1,-1)\to \mathcal{K}$ and $\mu$
is zero, then $\mu$ induces an exact sequence
\[\mathcal{O}(-3,-1)
\xrightarrow{\bar{\mu}} \mathcal{O}(-1,0)^{\oplus 2}
\to E_2^{-1,1}\to 0.\]
Since $\bar{\mu}\neq 0$, the composite
\[
\mathcal{O}(-3,-1)
\xrightarrow{\bar{\mu}} \mathcal{O}(-1,0)^{\oplus 2}
\to \mathcal{O}(-1,0)
\]
of $\bar{\mu}$ and some projection
$\mathcal{O}(-1,0)^{\oplus 2}\to \mathcal{O}(-1,0)$ is nonzero,
and thus its cokernel is isomorphic to $\mathcal{O}_D(-p)$
for some divisor $D$ of type $(2,1)$ and a point $p\in D$. 
Since this cokernel $\mathcal{O}_D(-p)$ is also a quotient of $E_2^{-1,1}$,
this contradicts the property of $E_2^{-1,1}$.
Hence the composite $\nu_a$ above is nonzero. 
Let $\mathcal{C}_a$ be the cokernel of $\nu_a$. 
Since $\nu_a$ gives rise to a section $(s_1,s_2)\in H^0(\mathcal{O}(0,1))^{\oplus 2}$,
we see that  
\[\mathcal{C}_a
\cong 
\begin{cases}
\mathcal{O}(-1,1)\quad \textrm{if } (s_1)_0\cap (s_2)_0=\emptyset;\\
\mathcal{O}(-1,0)\oplus \mathcal{O}_{L_1}(-1)\quad \textrm{if } (s_1)_0\cap (s_2)_0=L_1.
\end{cases}
\]
Note that $\mathcal{C}_a$ also 
fits in an exact sequence
\[
0\to E_2^{-2,1}\to \mathcal{O}(-3,-1)
\xrightarrow{\bar{\mu}_a} \mathcal{C}_a\to E_2^{-1,1}\to 0.
\]
We claim here that $\bar{\mu}_a$ is injective. To see this, first note that 
$\bar{\mu}_a$ is nonzero since $E_2^{-1,1}$ cannot admit 
a negative degree quotient on a line $L_1$.
If $\bar{\mu}_a$ is not injective, then $\mathcal{C}_a\cong 
\mathcal{O}(-1,0)\oplus \mathcal{O}_{L_1}(-1)$
and the composite of $\bar{\mu}_a$ and the projection
$\mathcal{O}(-1,0)\oplus \mathcal{O}_{L_1}(-1)
\to \mathcal{O}(-1,0)$ is zero, and thus $E_2^{-1,1}$ admits
$\mathcal{O}(-1,0)$ as a quotient. This is a contradiction.
Hence the claim holds: $\bar{\mu}_a$ is injective,
and thus 
\[E_2^{-2,1}=0.\]
The exact sequences~(\ref{GenSeq}),
(\ref{E3}), and (\ref{filtration}) then become the following exact sequences:
\begin{gather}
0\to \mathcal{O}(-1,-1)^{\oplus 4}
\to \mathcal{O}(-1,0)^{\oplus 2}
\oplus 
\mathcal{O}(0,-1)^{\oplus 2}
\to E_2^{-1,1}\to 0;\label{deg0Elliptic}\\
0\to \mathcal{O}^{\oplus r}\to \mathcal{E}\to E_2^{-1,1}\to 0.\label{deg0EllipticQuot}
\end{gather}
These two exact sequences result in  Case $(13)$ of Theorem~\ref{Chern(2,2)}.
Here we address the structure of $E_2^{-1,1}$ some more;
if $\mathcal{C}_a$ is isomorphic to 
$\mathcal{O}(-1,0)\oplus \mathcal{O}_{L_1}(-1)$,
then the composite of $\bar{\mu}_a$ and the projection
$\mathcal{O}(-1,0)\oplus \mathcal{O}_{L_1}(-1)
\to \mathcal{O}(-1,0)$ is nonzero,
and its cokernel is a quotient of $E_2^{-1,1}$ 
and isomorphic to $\mathcal{O}_D(-p)$ for some point $p$
on a divisor $D$ of type $(2,1)$. This is a contradiction.
Hence $\mathcal{C}_a\cong \mathcal{O}(-1,1)$,
and we see that 
\[E_2^{-1,1}\cong \mathcal{O}_E(\mathfrak{d})\]
for some divisor $\mathfrak{d}$ of degree $0$ on a divisor $E$ of type $(2,2)$.

\paragraph{Case $(b)$}\label{case(b)}
Since $\mu\neq 0$,
the composite 
\[\nu_b:\mathcal{O}(-2,-1)\to \mathcal{K}
\xrightarrow{\mu}
\mathcal{O}(-1,0)^{\oplus 2}\]
of a suitable inclusion $\mathcal{O}(-2,-1)\to \mathcal{K}$ and $\mu$
is nonzero. Let $\mathcal{C}_b$
be the cokernel of $\nu_b$.
Then $\mathcal{C}_b$ fits in an exact sequence
\[
0\to E_2^{-2,1}\to \mathcal{O}(-2,-1)
\xrightarrow{\bar{\mu}_b} \mathcal{C}_b\to E_2^{-1,1}\to 0.
\]
Since $\nu_b$ gives rise to a section $(t_1,t_2)\in H^0(\mathcal{O}(1,1))^{\oplus 2}$,
$\mathcal{C}_b$ satisfies one of the following:
\begin{enumerate}
\item[(i)] If $(t_1)_0\cap (t_2)_0=C$ for some divisor $C$ of type $(1,1)$,
then 
\[\mathcal{C}_b\cong \mathcal{O}(-1,0)\oplus \mathcal{O}_C(-p)\]
for some point $p\in C$;
\item[(ii)]  If $(t_1)_0\cap (t_2)_0=L_2$,
then $\mathcal{C}_b$ fits in an exact sequence
\[0\to \mathcal{O}_{L_2}(-1)\to \mathcal{C}_b\to \mathcal{O}(-1,1)\to 0;\]
\item[(iii)]  If $(t_1)_0\cap (t_2)_0=L_1$,
then $\mathcal{C}_b$ fits in an exact sequence
\[0\to \mathcal{O}_{L_1}(-2)\to \mathcal{C}_b\to \mathcal{O}\to 0;\]
\item[(iv)]  If $(t_1)_0\cap (t_2)_0=Z$ for some zero-dimensional subscheme 
of length $2$,
then 
\[\mathcal{C}_b\cong \mathcal{I}_Z(0,1).\]
\end{enumerate}

We will divide 
\ref{case(b)} into two cases:
in \ref{not injective}, we consider the case 
where $\bar{\mu}_b$ is not injective;
in \ref{injective}, we address the case where $\bar{\mu}_b$ is injective.

\subparagraph{
Suppose that $\bar{\mu}_b$
is not injective.}\label{not injective}
First consider Case (i). 
We see that the composite of $\bar{\mu}_b$ and the 
projection $\mathcal{C}_b\to \mathcal{O}(-1,0)$ is zero,
and thus $E_2^{-1,1}$ admits
$\mathcal{O}(-1,0)$ as a quotient. This is a contradiction.
Hence Case (i) does not occur.
Next consider Case (ii). 
Then the composite of $\bar{\mu}_b$ and the 
projection $\mathcal{C}_b\to \mathcal{O}(-1,1)$ is zero,
and $E_2^{-1,1}$ admits
$\mathcal{O}(-1,1)$ as a quotient,
which is a contradiction.
Hence Case (ii) does not occur.
For Case (iii),
we first observe that $\bar{\mu}_b:\mathcal{O}(-2,-1)\to \mathcal{C}_b$
factors through $\mathcal{O}_{L_1}(-2)$.
If the induced morphism $\mathcal{O}(-2,-1)\to \mathcal{O}_{L_1}(-2)$
is zero, then $\bar{\mu}_b=0$ and $E_2^{-1,1}\cong \mathcal{C}_b$.
We claim here that the exact sequence in Case (iii) splits.
Assume, to the contrary, that it does not split.
We will 
apply the Bondal spectral sequence~\eqref{BondalSpectral} to $\mathcal{C}_b$.
Set $E_2^{p,q}(\mathcal{C}_b):=\mathcal{H}^p(\Ext^q(G,\mathcal{C}_b)\lotimes_AG)$
in order 
to avoid confusion with $E_2^{p,q}$ for $\mathcal{E}$. 
Non-splitting of the exact sequence in (iii)
implies that $h^q(\mathcal{C}_b)=0$ for all $q$.
We see that $h^1(\mathcal{C}_b(-1,0))=2$,
that $h^q(\mathcal{C}_b(-1,0))=0$ for $q=0,2$,
that $h^1(\mathcal{C}_b(0,-1))=1$,
that $h^q(\mathcal{C}_b(0,-1))=0$ for $q=0,2$,
that $h^1(\mathcal{C}_b(-1,-1))=2$,
and that $h^q(\mathcal{C}_b(-1,-1))=0$ for $q=0,2$.
Hence $\Ext^2(G,\mathcal{C}_b)=0$, $\Hom(G,\mathcal{C}_b)=0$,
and $\Ext^1(G,\mathcal{C}_b)$ fits in the following exact sequence
of right $A$-modules:
\[
0\to S_1^{\oplus 2}\oplus S_2\to \Ext^1(G,\mathcal{C}_b)\to S_3^{\oplus 2}\to 0,
\]
which induces the following distinguished triangle:
\[
\mathcal{O}(-1,-1)^{\oplus 2}[1]\to
(\mathcal{O}(-1,0)^{\oplus 2}\oplus \mathcal{O}(0,-1))[1]
\to \Ext^1(G,\mathcal{C}_b)\lotimes_A G\to.
\]
Therefore we have the following exact sequence 
\[0\to E_2^{-2,1}(\mathcal{C}_b)\to 
\mathcal{O}(-1,-1)^{\oplus 2}\to
\mathcal{O}(-1,0)^{\oplus 2}\oplus \mathcal{O}(0,-1)
\to E_2^{-1,1}(\mathcal{C}_b)\to 0.
\]
Moreover we see that $E_2^{p,q}=0$ unless $(p,q)=(-2,1)$ or $(-1,1)$.
Thus we have $E_2^{-2,1}(\mathcal{C}_b)=E^{-1}(\mathcal{C}_b)=0$,
where $E^{-1}(\mathcal{C}_b)$ denotes the $E^{-1}$ for $\mathcal{C}_b$.
Now we see that $E_2^{-1,1}(\mathcal{C}_b)\cong \mathcal{C}_b\cong E_2^{-1,1}$.
Since $E_2^{-1,1}\vert_{L_2}$ cannot admit a negative degree quotient,
the exact sequence above induces the following exact sequence:
\[
0\to
\mathcal{O}(-2,-1)\to
\mathcal{O}(-1,0)^{\oplus 2}
\to E_2^{-1,1}\to 0,
\]
which implies that $E_2^{-1,1}$ admits a negative degree quotient on a
divisor of type $(1,1)$. This is a contradiction.
Hence the exact sequence in Case (iii) splits 
if $\bar{\mu}_b=0$.
Then $E_2^{-1,1}$
admits $\mathcal{O}_{L_1}(-2)$ as a quotient,
which is a contradiction. Hence the induced morphism 
$\mathcal{O}(-2,-1)\to \mathcal{O}_{L_1}(-2)$ is nonzero and thus surjective,
and the following conclusions are drawn:  
\begin{align*}
E_2^{-2,1}&\cong \mathcal{O}(-2,-2);& 
E_2^{-1,1}&\cong \mathcal{O}.
\end{align*}
The exact sequences (\ref{E3}) and (\ref{filtration}) then become the following 
exact sequences:
\begin{align}
0\to \mathcal{O}(-2,-2)\to &\mathcal{O}^{\oplus r}\to E_3^{0,0}\to 0;
\label{E3o}
\\
0\to E_3^{0,0}\to &\mathcal{E}\to \mathcal{O}\to 0.\label{Ocoker}
\end{align}
These two exact sequences yield Case $(12)$ of Theorem~\ref{Chern(2,2)}.
Finally consider Case (iv).
In this case, 
$\bar{\mu}_b=0$, and thus $\mathcal{C}_b\cong E_2^{-1,1}$.
Then $E_2^{-1,1}\vert_C$ admits $\mathcal{O}_C(-1)$ as a quotient for some 
divisor $C$ of type $(1,1)$ passing through $Z$. This is a contradiction.
Hence Case (iv) does not occur.
\begin{remark}\label{SupportO}
If $\mathcal{E}$ fits in the exact sequences~\eqref{E3o} and \eqref{Ocoker},
then the cokernel of 
the evaluation map 
$H^0(\mathcal{E})\otimes \mathcal{O}\to \mathcal{E}$ 
is $\mathcal{O}$.
\end{remark}

\subparagraph{
Suppose that $\bar{\mu}_b$ is injective.}\label{injective} 
Then 
\[E_2^{-2,1}=0.\]
As in Case $(a)$, 
this implies that $E_2^{-1,1}$
fits in the exact sequences 
(\ref{deg0Elliptic}) and (\ref{deg0EllipticQuot})
and we obtain Case $(13)$ of Theorem~\ref{Chern(2,2)}.
Here we address the structure of $E_2^{-1,1}$ some more;
if $\mathcal{C}_b$ belongs to Case (i), then $E_2^{-1,1}$ admits 
$\mathcal{O}_{D_1}(-1)$ as a quotient
for some divisor $D_1$ of type $(1,1)$, which is a contradiction.
Hence Case (i) does not occur. 
If $\mathcal{C}_b$ belongs to Case (ii), then $E_2^{-1,1}$ admits 
$\mathcal{O}_{D_2}(-1)$ as a quotient 
for some divisor $D_2$ of type $(1,2)$, which is a contradiction.
Hence Case (ii) does not occur. 
If $\mathcal{C}_b$ belongs to Case (iii), then $E_2^{-1,1}$ fits in the following 
exact sequence:
\[
0\to \mathcal{O}_{L_1}(-2)\to E_2^{-1,1}\to \mathcal{O}_{D_3}\to 0,
\] 
where $D_3$ is a divisor of type $(2,1)$.
In this case, $E_2^{-1,1}$ lies on the divisor $L_1+D_3$ of type $(2,2)$.
Finally if $\mathcal{C}_b$ belongs to Case (iv),
then $E_2^{-1,1}$ fits in the following 
exact sequence:
\[
0\to \mathcal{O}(-2,-1)\to \mathcal{I}_{Z}(0,1)\to E_2^{-1,1}\to 0,
\] and thus 
\[
E_2^{-1,1}\cong \mathcal{O}_E(\mathfrak{d})
\]
for some divisor $\mathfrak{d}$ of degree $0$
on an effective divisor $E$ of type $(2,2)$. 

\begin{remark}\label{support(2,2)}
If $\mathcal{E}$ fits in the exact sequences~$(\ref{deg0Elliptic})$
and $(\ref{deg0EllipticQuot})$, then the cokernel of 
the evaluation map 
$H^0(\mathcal{E})\otimes \mathcal{O}\to \mathcal{E}$ 
is supported on a divisor $E$ of type $(2,2)$.
\end{remark}

\subsection{The case $h^1(\mathcal{E})>0$}\label{h^1>0}
Recall that $c_2\leq c_1^2=8$ since $\mathcal{E}$ is nef.
The assumption $h^1(\mathcal{E})>0$ then implies that $c_2=c_1^2=8$
by 
Lemma~\ref{h^1EvanishingCondition}.
The natural isomorphism $H^1(\mathcal{E})\cong \Ext^1(\mathcal{O},\mathcal{E})$
induces the following exact sequence:
\[0\to \mathcal{E}\to \mathcal{E}_0\to 
H^1(\mathcal{E})\otimes\mathcal{O} \to 0.\]
It follows from \cite[Theorem~6.2.12 (ii)]{MR2095472}
that $\mathcal{E}_0$ is a nef vector bundle of rank 
$r+h^1(\mathcal{E})$ 
with $c_1(\mathcal{E}_0)=(2,2)$, $c_2(\mathcal{E}_0)=8$,
and $h^1(\mathcal{E}_0)=0$.
Note here the evaluation map 
$\ev:H^0(\mathcal{E}_0)\otimes \mathcal{O}\to \mathcal{E}_0$ 
factors through $\mathcal{E}$ since $H^0(\mathcal{E})\cong H^0(\mathcal{E}_0)$.
Hence we have a surjection 
$\Coker(\ev)\to
H^1(\mathcal{E})\otimes\mathcal{O}$.
From Remarks~\ref{SupportO} and \ref{support(2,2)},
it follows that $\Coker(\ev)\cong \mathcal{O}$, that $h^1(\mathcal{E})=1$,
and that $\mathcal{E}$ is globally generated.
Thus $\mathcal{E}$ belongs to Case $(12)$ of Theorem~\ref{Chern(2,2)}.
This completes the proof of Theorem~\ref{Chern(2,2)}.

\bibliographystyle{alpha}

\newcommand{\noop}[1]{} \newcommand{\noopsort}[1]{}
  \newcommand{\printfirst}[2]{#1} \newcommand{\singleletter}[1]{#1}
  \newcommand{\switchargs}[2]{#2#1}

\end{document}